\documentclass[12pt]{article}
\usepackage{no-ejc}
\usepackage{ytableau}
\usepackage{tikz}
\usetikzlibrary{matrix}
\usepackage{blkarray}
\usepackage{amsmath}
\usepackage{listings}
\usepackage{mathtools}

\usepackage{amsmath,amssymb}
\usepackage{graphicx}
\usepackage{enumerate}   
\usepackage[colorlinks=true,citecolor=black,linkcolor=black,urlcolor=blue]{hyperref}
\usepackage{cancel}

\DeclareMathOperator*{\argmax}{arg\,max}

\DeclareMathOperator*{\argmin}{arg\,min}

\dateline{}{}{}


\usepackage{hyperref}
\hypersetup{
    colorlinks,
    citecolor=black,
    filecolor=black,
    linkcolor=black,
    urlcolor=black
}

\usepackage{tikz}
\usetikzlibrary{arrows}
\usetikzlibrary{shapes,snakes}
\usepackage{geometry}
\makeatletter
\def\hlinewd#1{%
  \noalign{\ifnum0=`}\fi\hrule \@height #1 \futurelet
   \reserved@a\@xhline}
\makeatother

\Copyright{  The authors. Released under the CC BY-ND license (International 4.0).}

\title{Jack Derangements}

\author{	
	Nathan Lindzey\\
	\small Department of Computer Science\\[-0.8ex]
	\small Technion -- Israel Institute of Technology\\[-0.8ex] 
	\small\tt lindzeyn@gmail.com
}

\begin{document}
\maketitle

\begin{abstract}
For each integer partition $\lambda \vdash n$ we give a simple combinatorial expression for the sum of the Jack character $\theta^\lambda_\alpha$ over the integer partitions of $n$ with no singleton parts. 
For $\alpha = 1,2$ this gives closed forms for the eigenvalues of the permutation and perfect matching derangement graphs, resolving an open question in algebraic graph theory. 
A byproduct of the latter is a simple combinatorial formula for the immanants of the matrix $J-I$ where $J$ is the all-ones matrix, which might be of independent interest.
Our proofs center around a Jack analogue of a hook product related to Cayley's $\Omega$--process in classical invariant theory, which we call \emph{the principal lower hook product}.
\end{abstract}

\section{Introduction}\label{sec:intro}

Let $x := x_1,x_2,\cdots$ be an infinite set of indeterminates and let $\alpha \in \mathbb{R}$ be a real parameter. The \emph{(integral form) Jack polynomials} $J_\lambda := J_\lambda(x;\alpha )$ are defined as the unique basis $\{J_\lambda\}$ for the ring of symmetric functions that satisfies the following properties.
\begin{itemize}
	\item \emph{Orthogonality:} $\langle J_\lambda, J_\mu \rangle_\alpha = 0$ if $\lambda \neq \mu$ where $\langle \cdot,\cdot \rangle_\alpha$ is the \emph{deformed Hall inner product} defined on the \emph{power sum basis} $\{ p_\lambda\}$ such that $\langle p_\lambda , p_\mu \rangle_\alpha := \delta_{\lambda,\mu} \alpha^{\ell(\lambda)} z_\lambda$.
	\item \emph{Triangularity:} $J_\lambda = \sum_{\mu \trianglelefteq \lambda} c_{\lambda \mu} m_\mu$ where $\{m_\mu\}$ is the \emph{monomial basis} and $\trianglelefteq$ denotes the \emph{dominance ordering} on integer partitions $\lambda \vdash n$.
	\item \emph{Normalization:} $[m_{1^n}]J_\lambda = n!$.
\end{itemize}
\noindent We refer the reader to~\cite[Ch.~IV \textsection 10]{Macdonald95} and~\cite{Stanley89} for a detailed treatment of these polynomials. In this work, we restrict our attention to the power sum expansion of the Jack polynomials
\[
	J_\lambda = \sum_{\mu \vdash n} \theta^{\lambda}_\alpha(\mu) p_\mu \quad \text{ for all } \lambda \vdash n.
\]
The $\theta^{\lambda}_\alpha$'s are called the \emph{Jack characters} because they are a deformation of a normalization of the \emph{irreducible characters} $\chi^\lambda$ of \emph{the symmetric group} $S_n$. In particular, the Jack polynomials at $\alpha = 1,2$ recover the integral forms of the \emph{Schur} and \emph{Zonal polynomials} respectively. These specializations have been widely studied in algebraic combinatorics due to their connections with $S_n$ and the set $\mathcal{M}_{2n}$ of \emph{perfect matchings of the complete graph} $K_{2n}$, but for arbitrary $\alpha \in \mathbb{R}$ many open questions remain~\cite{AlexanderssonHW21,Stanley89,Macdonald95}. This state of affairs has led to an investigation of the Jack characters since they provide \emph{dual} information about Jack polynomials that may shed light on these open questions; however, the dual path towards understanding Jack polynomials is paved with its own conjectures~\cite{Hanlon88,Lassalle98,Lassalle09}. 
We make modest progress in this direction by considering \emph{sums} of $\theta_\alpha^\lambda(\mu)$'s rather than single $\theta_\alpha^\lambda(\mu)$'s.

Let $\text{fp}(\mu)$ be the number of singleton parts of $\mu$. Define the \emph{$\lambda$-Jack derangement sum}
\[
	\eta^\lambda_\alpha := \sum_{\substack{\mu \vdash n \\ \text{fp}(\mu) = 0}} \theta^{\lambda}_\alpha(\mu)
\]
to be the sum of the Jack character $\theta^\lambda_\alpha$ over the \emph{derangements}, i.e., partitions $\mu \vdash n$ with no singleton parts. To motivate this definition, recall that if $\lambda \vdash n$ is the cycle type of a permutation $\pi \in S_n$, then $\pi$ is a \emph{derangement} if and only if $\text{fp}(\lambda) = 0$. Let $D_n \subseteq S_n$ be the set of derangements of $S_n$. One can show that $\eta^{\lambda}_1$ is a scaled character sum over $D_n$, i.e.,
$$\eta^{\lambda}_1 = \sum_{\substack{\mu \vdash n \\ \text{fp}(\mu) = 0}} \theta^{\lambda}_1(\mu) =\sum_{\substack{\mu \vdash n \\ \text{fp}(\mu) = 0}} \frac{|C_\mu|}{\chi^\lambda(1)} \chi^\lambda(\mu) = \frac{1}{\chi^\lambda(1)}\sum_{\pi \in D_n} \chi^\lambda(\pi) $$ 
where $C_\mu \subseteq S_n$ is the conjugacy class corresponding to $\mu \vdash n$. For $\alpha = 2$, an analogous result holds for the so-called \emph{perfect matching derangements} of $\mathcal{M}_{2n}$ (see~\cite{Lindzey17}, for example). We are unaware of any combinatorial models for $\alpha \neq 1,2$, but it is natural to view $\eta^\lambda_\alpha$ as the $\alpha$-analogue of the character sum over derangements, which is the main focus of this paper.

While little is known about the Jack derangement sums for arbitrary $\alpha \in \mathbb{R}$, the $\alpha = 1,2$ cases have received special attention in algebraic graph theory because they are in fact the eigenvalues of the so-called \emph{derangement graphs}.

\begin{itemize}
\item The set $\{\eta^\lambda_1\}_{\lambda \vdash n}$ is the spectrum of the \emph{permutation derangement graph}
\[
	\Gamma_{n,1} = (S_n,E) \text{ where } \pi\sigma \in E \Leftrightarrow \sigma \pi^{-1} \in D_n,
\]
i.e., the normal Cayley graph of $S_n$ generated by $D_n$. 
See~\cite[Ch.~14]{GodsilMeagher} or \cite{Renteln07} for more details on the permutation derangement graph.

\item The set $\{\eta^\lambda_2\}_{\lambda \vdash n}$ is the spectrum of the \emph{perfect matching derangement graph}
\[
	\Gamma_{n,2} = (\mathcal{M}_{2n},E) \text{ where } mm' \in E \Leftrightarrow m \cap m' = \emptyset. 
\]
See~\cite[Ch.~15]{GodsilMeagher} or \cite{Lindzey17} for more details on the perfect matching derangement graph.
\end{itemize}
These graphs made their debut in \emph{Erd\H{o}s--Ko--Rado combinatorics}, a branch of extremal combinatorics that studies how large families of combinatorial objects can be subject to the restriction that any two of its members intersect. By design, the \emph{independent sets} (sets of vertices that are pairwise non-adjacent) of $\Gamma_{n,\alpha}$ are in one-to-one correspondence with the so-called \emph{intersecting families} of permutations and perfect matchings for $\alpha = 1,2$, and the spectra of these graphs have been used to give tight upper bounds and characterizations of the largest intersecting families of $S_n$ and $\mathcal{M}_{2n}$. We refer the reader to~\cite{GodsilMeagher} for a comprehensive account of algebraic techniques in Erd\H{o}s--Ko--Rado combinatorics.
 
The derangement graphs are interesting in their own right since they are natural analogues of the celebrated \emph{Kneser graph}\footnote{Recall that the \emph{Kneser graph} is the graph defined on $k$-sets of $\{1,2,\cdots,n\}$ such that two $k$-sets are adjacent if they are disjoint.}, a cornerstone of algebraic graph theory~\cite[Ch.~7]{GodsilRoyle}. Because the algebraic combinatorics of permutations and perfect matchings are more baroque than that of subsets, the eigenvalues of the derangement graphs have proven to be far more challenging to understand. The following is a brief overview of the results in this area. 
 
The first non-trivial recursion for the eigenvalues of the permutation derangement graph was derived by Renteln~\cite{Renteln07} using determinantal formulas for the \emph{shifted Schur functions}~\cite{OkunkovO97}, which he used to calculate the minimum eigenvalue of the permutation derangement graph. Using different techniques, Ellis~\cite{Ellis12} later computed the minimum eigenvalue of the permutation derangement graph. Deng and Zhang~\cite{DengZ11} determined the second largest eigenvalue. In~\cite{KuW10}, Ku and Wales investigated some interesting properties of the eigenvalues of the permutation derangement graph. In particular, they proved \emph{The Alternating Sign Theorem}, namely, that $\text{sgn}~\eta^\lambda_1 = (-1)^{|\lambda|-\lambda_1}$ for all $\lambda$, and they offered a conjecture on the magnitudes of the eigenvalues known as the \emph{Ku--Wales Conjecture}. In~\cite{KuW13a}, Ku and Wong proved this conjecture by deriving another recursive formula using shifted Schur functions that also led to a simpler proof of the Alternating Sign Theorem.

It was soon noticed that the algebraic properties of the perfect matching derangement graph parallel those of the permutation derangement graph. The minimum eigenvalue of the perfect matching derangement graph was computed by Godsil and Meagher~\cite{GodsilM15} and later by Lindzey~\cite{LindzeyPhd,Lindzey20}. An analogue of the Alternating Sign Theorem was conjectured in~\cite{Lindzey17,GodsilMeagher} which was recently proven by both Renteln~\cite{Renteln21} and Koh et al~\cite{KohKW23}. In an earlier effort to prove this conjecture, Ku and Wong~\cite{KuW18} give recursive formulas for $\eta^\lambda_2$ and a few closed forms for select shapes. In~\cite{Murali20}, Srinivasan gives more computationally efficient formulas for the eigenvalues of the perfect matching derangement graph. Godsil and Meagher ask whether an analogue of the Ku--Wales conjecture holds for the perfect matching derangement graph~\cite[pg.~316]{GodsilMeagher}. The latter has remained open since the eigenvalues of the perfect matching derangement graph have defied nice recursive expressions akin to permutation derangement graph. This is because the aforementioned determinantal formulas for shifted Schur functions do not exist for shifted Zonal polynomials or shifted Jack polynomials. 

The main shortcoming of the known eigenvalue formulas for the derangement graphs is that they cannot be evaluated efficiently, i.e., they lack ``good formulas". Indeed, finding closed forms for these eigenvalues was deemed a difficult open problem~\cite[pg. 316]{GodsilMeagher}, perhaps due to the formal hardness of evaluating the irreducible characters of the symmetric group~\cite{PakP17,IkenmeyerPP22,Pak22}. 
Our results show that good formulas for these eigenvalues exist.  

To state our main results we need a few definitions. Let $h^\lambda_*(i,j) :=  \alpha a_\lambda(i,j) + l_\lambda(i,j)  + 1$ be the \emph{lower hook length} of the cell $(i,j) \in \lambda$ where $a_\lambda(i,j)$ and $l_\lambda(i,j)$ denote \emph{arm length} and \emph{leg length} respectively (see Section~\ref{sec:jack} for definitions). We define
\[
  H_*^1(\lambda)  := h^\lambda_*(1,1) h^\lambda_*(1,2) \cdots h^\lambda_*(1,\lambda_1)
\]
to be the \emph{principal lower hook product} of the integer partition $\lambda$. For $\alpha = 1$, the lower hook length is just the usual notion of hook length, in which case we call $H^1_*(\lambda)$ the \emph{principal hook product}.  Note that the principal hook product for $\lambda = (n)$ is simply $n!$. 

It turns out that the principal hook product for arbitrary $\lambda$ arises naturally in classical invariant theory, namely, in the evaluation of a differential operator known as \emph{Cayley's $\Omega$--process} (see~\cite{CanfieldW81}). Independently, Filmus and Lindzey~\cite{FilmusL22a} observe a similar phenomenon in their study of harmonic polynomials on perfect matchings, wherein they show that the principal lower hook product appears in the evaluation of a family of differential operators acting polynomial spaces associated with perfect matchings. From the results of~\cite{FilmusL22a}, we show in Section~\ref{sec:php} that the principal hook product $H^1_*(\lambda)$ counts an interesting class of colored permutations $\mathcal{S}_\lambda$, defined as follows. 

For each $i \in [n] := \{1,2,\ldots,n\}$, we assign a list of colors $L(i) \subseteq [m]$ for some $m \in \mathbb{N}$. We define a \emph{colored permutation} $(c,\sigma)$ to be an assignment of colors $c = c_1,c_2,\ldots, c_n$ such that $c_i \in L(i)$ and a permutation $\sigma \in \text{Sym}([n])$ such that $\sigma(i) = j \Rightarrow c_i = c_j$, i.e., each cycle of the permutation is monochromatic. Any partition $\lambda$ defines a color list on each element $i$ of the symbol set $[\lambda_1]$ by setting $L(i) := [\lambda'_i]$ where $\lambda'$ denotes the \emph{transpose} or \emph{conjugate} partition of $\lambda$. We define $\mathcal{S}_\lambda$ to be the set of all such colored permutations, formally, 
\[
	\mathcal{S}_\lambda := \{ (c \in [\lambda'_1] \times \cdots \times [\lambda'_{\lambda_1}], \sigma \in S_{\lambda_1}) : \sigma(i) = j \Rightarrow c_i = c_j \text{ for all } i \in [\lambda_1]\}.
\]
We say that a colored permutation $(c,\sigma) \in \mathcal{S}_\lambda$ is a \emph{derangement} if $\sigma(i) = i \Rightarrow c_i \neq 1$ for all $1 \leq i \leq \lambda_1$. In other words, these are the colored permutations that have no colored cycles in common with $(1,\ldots,1, ()) \in \mathcal{S}_\lambda$. Let $\mathcal{D}^\lambda$ be the set of derangements of $\mathcal{S}_\lambda$, and let $\mathcal{D}^\lambda_{k}$ be the set of derangements of $\mathcal{S}_\lambda$ with exactly $k$ disjoint cycles. We define $D^\lambda := |\mathcal{D}^\lambda|$ and $d^\lambda_k := |\mathcal{D}^\lambda_k|$, so that 
\[
	D^\lambda = d^\lambda_1 + d^\lambda_2 + \cdots + d^\lambda_{\lambda_1}.
\]
For $\lambda = (n)$, the $d^\lambda_k$'s recover the \emph{(unsigned) associated Stirling numbers of the first kind}, i.e., the number of derangements of $S_n$ that have precisely $k$ disjoint cycles (see~\cite[pg.~256]{Comtet74}). 
\emph{Colored perfect matchings} $\mathcal{M}_\lambda$ and their derangements $\mathcal{D}_\lambda'$ can be defined in a similar but slightly more complicated manner, which we defer to Section~\ref{sec:php}. 

For any $\alpha \in \mathbb{R}$, we define 
\[
	D^\lambda_\alpha  := \sum_{k=1}^{\lambda_1} d^\lambda_k \alpha^{\lambda_1 - k}
\]
to be the \emph{$\lambda$-Jack derangement number}. 

Our first main result is that the Jack derangement sums equal the Jack derangement numbers (up to sign).
\begin{theorem}\label{thm:main} For any shape $\lambda$ and $\alpha \in \mathbb{R}$, we have
\[
	\eta_\alpha^\lambda = (-1)^{|\lambda| - \lambda_1} D^\lambda_\alpha
\]
\end{theorem}
\noindent Theorem~\ref{thm:main} gives simpler, unified, and more general proofs of all the aforementioned results on the derangement graphs, which we list below. 
\begin{corollary}[Alternating Sign Theorem]\label{thm:ast} For any shape $\lambda$ and $\alpha \geq 0$, we have
\[
	\emph{sgn}~\eta_\alpha^\lambda = (-1)^{|\lambda| - \lambda_1}.
\]
\end{corollary}
\begin{corollary}[Ku--Wales Theorem]\label{thm:kwt}
For all $\mu, \lambda \vdash n$ such that $\mu_1 = \lambda_1$ and $\alpha \geq 0$, we have
$$\mu \trianglelefteq \lambda  \Rightarrow |\eta_\alpha^\mu| \leq |\eta_\alpha^\lambda|.$$
\end{corollary}
\noindent Setting $\alpha = 2$ in Corollary~\ref{thm:kwt} answers Godsil and Meagher's question on the Ku--Wales conjecture for the perfect matching derangement graph~\cite[pg. 316]{GodsilMeagher}.
\begin{corollary}\label{thm:minmax}
	For all $\alpha \geq 1$ and $n \geq 6$, we have 
	$$(n) = \argmax_{\lambda \vdash n}~\eta_\alpha^\lambda, \quad (n-1,1) = \argmin_{\lambda \vdash n}~\eta_\alpha^\lambda, \quad \text{ and } \quad (n-1,1) = \argmax_{\substack{\lambda \vdash n \\ \lambda \neq (n)}}~|\eta_\alpha^\lambda|.$$
\end{corollary}

Finally, we note that colored permutations have appeared before in the study of the character theory of the symmetric group. In~\cite{Stanley06} Stanley conjectures a formula for $\theta^\lambda_1(\mu)$ in terms of colored permutations, which was later proven by F\'eray~\cite{Feray06} and reformulated by F\'eray and \'Sniady~\cite{FerayS11}. The combinatorics involved in their reformulation makes the expression more amenable to asymptotic analysis, leading to sharper results on the asymptotic character theory of $S_n$~\cite{FerayS11}. Moreover, Lasselle~\cite{Lassalle07} conjectures that an analogue of the Stanley--F\'eray formula holds for the Jack characters. Although these works all feature colored permutatons, the main results center around their factorizations and asymptotics, which we do not consider. It is also not clear how to use the Stanley--F\'eray--\'Sniady formula to recover our main results for $\alpha = 1$. At any rate, these works and the present show that colored permutations have an understated role in the character theory of the symmetric group that seems worthy of future investigation.

\subsection*{Organization}

The paper is organized as follows. In Section~\ref{sec:prelim} we overview basic terminology and definitions in the theory of symmetric functions. We introduce the shifted Jack polynomials in Section~\ref{sec:jack} and show that the Jack derangement sums can be written as an alternating sum of shifted Jack polynomials (Theorem \ref{thm:jack}). The expression we obtain is difficult to work with, so in Section~\ref{sec:php} we cover some combinatorial results of~\cite{AlexanderssonF17,FilmusL22a} that lead to a simpler combinatorial formulation of the expression (Corollary~\ref{cor:eigs}). This combinatorial expression is used to prove the main result for $\alpha = 0$ in Section~\ref{sec:0}, but it is still not explicit enough to obtain closed-form expressions for $\alpha \neq 0$. In Section~\ref{sec:minors} we prove a few technical lemmas about so-called \emph{minors} of principal lower hook products, which leads us to a more explicit formulation of a result of Alexandersson and F\'eray~\cite[Theorem 5.12]{AlexanderssonF17} in the language of finite differences. With these lemmas in hand, we prove our first main result (Theorem~\ref{thm:main}) in Section~\ref{sec:main}. 

The remainder focuses on various corollaries and specializations of $\alpha$ and $\lambda$. 
In Section~\ref{sec:cor}, we give short proofs for Corollary~\ref{thm:ast}, Corollary~\ref{thm:kwt}, and Corollary~\ref{thm:minmax}.  In Sections~\ref{sec:eigs} and \ref{sec:eigs2} we take a closer look at the $\alpha = 1,2$ case and prove our second main result, namely, closed tableau-theoretic expressions for the eigenvalues of the derangement graphs (Theorem~\ref{thm:eigen1} and Theorem~\ref{thm:eigen2}). For $\alpha=1$, this extends a result of Okazaki~\cite[Corollary 1.3]{OkazakiPhD} and Stanley~\cite[Ex.~7.63b]{StanleyV201} for hooks to arbitrary shapes, which can be reformulated as a result on immanants of the matrix $J-I$ (see Section~\ref{sec:eigs}). For $\alpha = 1$, we connect our closed form to Renteln's determinantal formula~\cite[Theorem 4.2]{Renteln07} for the eigenvalues of the permutation derangement graph.
We conclude with some open questions and directions for future work.

\section{Preliminaries}\label{sec:prelim}
The reader familiar with the theory of symmetric functions may skip this section, as our notation is completely standard, following Macdonald~\cite{Macdonald95} and Stanley~\cite{Stanley89}.

A \emph{shape} is a collection of cells $(i,j)$ such that $i,j \in \mathbb{N}_+$. We let $\lambda$ be an integer partition and we refer to it as a shape when we are appealing to its tableau interpretation. Let $\lambda / \mu$ denote the \emph{skew shape} obtained by deleting the cells of $\lambda$ that correspond to the partition $\mu$. Let $|\lambda|$ denote the \emph{size} of $\lambda$, i.e., the number of cells of $\lambda$. Let $\ell(\lambda)$ denote the \emph{length} of $\lambda$, i.e., the number of parts of $\lambda$. A \emph{Young tableau} $t$ of shape $\lambda$ is a tableau whose cells are labeled with the integers $[|\lambda|]$.  For any Young tableau $t$ let $t_{i,j}$ denote the entry of the $(i,j)$ cell of $t$. A Young tableau with entries strictly increasing along rows and columns is called \emph{standard}. Let $z_\lambda := 1^{m_1}2^{m_2} \cdots m_1! m_2! \cdots $ where $m_i$ is the number of parts of $\lambda$ equal to $i$. 

The \emph{elementary symmetric functions} are defined such that
\[
	e_\lambda := e_{\lambda_1} \cdots e_{\lambda_{\ell(\lambda)}} \quad \text{ where } \quad e_k(x_1,x_2,\cdots) = \sum_{i_1 < i_2 < \cdots < i_k} x_{i_1} \cdots x_{i_k}.
\]
The \emph{power sum symmetric functions} are defined such that
\[
	p_\lambda := p_{\lambda_1} \cdots p_{\lambda_{\ell(\lambda)}} \quad \text{ where } \quad p_k(x_1,x_2,\cdots) = \sum_{i=1} x_i^k.
\]
For a more detailed discussion of these polynomials, we refer the reader to~\cite[Ch.~I]{Macdonald95}.

We now review a well-known tableau-theoretic definition of the dominance ordering $\trianglelefteq$ on partitions. A cell $\Box \in \lambda$ is an \emph{outer corner} if the diagram obtained by removing $\Box$ is a partition of $|\lambda|-1$. A non-cell $(i,j)$ is an \emph{inner corner} if the diagram obtained by adding the cell $\Box := (i,j)$ to $\lambda$ is a partition of $|\lambda|+1$. For example, in the figure below, the outer corners are labeled ``$+$" and the inner corners are labeled ``-":
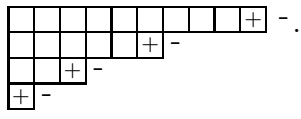
\[
\ytableausetup{smalltableaux}
\begin{ytableau}
~ & ~ & ~ & ~ & ~ & ~ & ~ & ~ & ~ & + & \none[\text{~~-~}]\\
~ & ~  &~ &~ & ~ & + & \none[\text{-}] \\ 
~ & ~ &+& \none[\text{-}]  \\
+ & \none[\text{-}]\\
\none[\text{-}]
\end{ytableau}.
\ytableausetup{nosmalltableaux}
\]
We write $\mu \nearrow \nu$ if $\nu$ can be obtained from $\mu$ by removing a single outer corner $\Box \in \mu$ and placing it on an inner corner of $\mu$ that lies in a row above $\Box$ in $\mu$. The following proposition is well-known.
\begin{proposition}\label{prop:dom}
	For any partitions $\mu,\lambda \vdash n$, we have $\mu \trianglelefteq \lambda$ if and only if there exists a sequence of partitions $\nu^1, \nu^2, \ldots, \nu^k$ such that 
	\[
	\mu = \nu^1 \nearrow \nu^2 \nearrow \cdots  \nearrow \nu^k = \lambda.
	\] 
\end{proposition}

\section{Shifted Jack Polynomials}\label{sec:jack}

We briefly review some of the standard terminology associated with Jack polynomials defined in the introduction. For any cell $(i,j) \in \lambda$, the \emph{leg length} $l_\lambda(i,j)$ of $(i,j)$ is the number of cells below $(i,j)$ in the same column of $\lambda$, and the \emph{arm length} $a_\lambda(i,j)$ of $(i,j)$ is the number of cells to the right of $(i,j)$ in the same row of $\lambda$, i.e., 
$$a_\lambda(i,j) = |\{ (i,k) \in \lambda : k > j\}| \quad \text{ and } \quad l_\lambda(i,j) = |\{ (k,j) \in \lambda : k > i\}|.$$ 
Note that arm length and leg length remain well-defined even when $\lambda$ is replaced by a set of cells that does not form an integer partition. Let
\[
	h^\lambda_*(i,j) := \alpha a_\lambda(i,j) + l_\lambda(i,j) + 1 \quad \text{ and } \quad h_\lambda^*(i,j) := \alpha (a_\lambda(i,j) + 1) + l_\lambda(i,j)
\]
be the \emph{lower hook length} and \emph{upper hook length} of $(i,j) \in \lambda$, respectively. Let
\[
	H_*^\lambda = \prod_{(i,j) \in \lambda} h^\lambda_*(i,j)  \quad \text{ and } \quad 
 H^*_\lambda = \prod_{(i,j) \in \lambda} h_\lambda^*(i,j)
\]
be the \emph{lower hook product} and \emph{upper hook product} of $\lambda$, respectively. Note that the lower and upper hook product remain well-defined even when $\lambda$ is replaced by a set of cells that does not form an integer partition.

A \emph{reverse semistandard Young tableau} is a tableau on $n$ cells with entries in $[n]$ that are weakly decreasing along rows and strictly decreasing down columns. Let $\mathrm{RSSYT}(\mu)$ be the set of all reverse semistandard Young tableau of shape $\mu$. For any reverse semistandard Young tableau $t$ of shape $\mu$, define  
$$\psi_t(\alpha) := \prod_{i=1}^{|\mu|} \psi_{\rho^i / \rho^{i-1}}(\alpha); \quad \psi_{\lambda/\mu}(\alpha) := \!\!\!\! \prod_{(i,j) \in R_{\lambda/\mu} \setminus C_{\lambda/\mu} } \!  \frac{
h^\mu_*(i,j)h_\lambda^*(i,j)}{h_\mu^*(i,j) h^\lambda_*(i,j)}
$$
where $\rho^i / \rho^{i-1}$ defines the skew shape in $t$ induced by the cells labeled $i$, and $R_{\lambda/\mu}$ ($C_{\lambda/\mu}$) denotes the set of boxes in a row (column) that intersects the shape $\lambda / \mu$. 

It is well-known that the \emph{(integral form) shifted Jack polynomials} are a basis for the ring of \emph{shifted symmetric functions}, i.e., the ring of polyomials symmetric in the variables $x_i - i/\alpha$, and are defined as follows:
$$J^\star_\lambda(x;\alpha)  :=  H^\lambda_* P^\star_\lambda(x;\alpha); \quad P^\star_\mu(x;\alpha) := \!\!\!\!\!\! \sum_{t \in \mathrm{RSSYT}(\mu)} \!\!\!\! \psi_t(\alpha) \prod_{(i,j) \in \mu} (x_{t_{i,j}} -  \bar{a}_\mu(i,j) + \bar{l}_\mu(i,j)/\alpha)$$
where $\bar{a}_\mu(i,j) = |\{ (i,k) \in \lambda : k < j\}|$ and $\bar{l}_\mu(i,j) = |\{ (k,j) \in \lambda : k < i\}|$ denote the \emph{co-arm length} and \emph{co-leg length} of $(i,j) \in \mu$, respectively. The polynomials $P^\star_\lambda$ are sometimes referred to as the \emph{normalized shifted Jack polynomials}.
 
Theorem~\ref{thm:jack} is a simple but opaque expression for $\eta^\lambda_\alpha$ in terms of shifted Jack polynomials.
These expressions are already known for $\eta^\lambda_1$ and $\eta^\lambda_2$ in terms of the determinantal formula for the shifted Schur polynomials~\cite{Renteln07} and recently for the shifted Zonal polynomials~\cite{Renteln21}. Theorem~\ref{thm:jack} is simply the Jack analogue of these results. Henceforth, we let $J_k^\star := J_{(k)}^\star$.

\begin{theorem}\label{thm:jack} For all $\lambda$ and $\alpha \in \mathbb{R}$, we have 
$$\eta^\lambda_\alpha =  \sum_{k=0}^{|\lambda|} \frac{(-1)^{|\lambda|-k}}{k!}   J_k^\star(\lambda).$$
\end{theorem}
\begin{proof}
We begin with the Cauchy formula for Jack polynomials~\cite{Stanley89} and its expansion:
\[
	\sum_\lambda \frac{J_\lambda(x;\alpha) J_\lambda(y;\alpha)}{H^*_\lambda H_*^\lambda} = \prod_{i,j}(1 - x_iy_j)^{1/\alpha} = \prod_{i \geq 1} \exp \left( \frac{p_i(x)p_i(y)}{\alpha i} \right).
\]
Since $\eta^\lambda_\alpha = J_\lambda |_{p_1 = 0 , p_2 = p_3 = \cdots = 1}$, setting $p_1(x) = 0$ and the remaining $p_i(x) = 1$ gives
\[
	\sum_\lambda  \frac{ \eta^{\lambda}_{\alpha} J_\lambda(y;\alpha)}{H^*_\lambda H_*^\lambda} = \prod_{i \geq 2} \exp \left( \frac{p_i(y)}{\alpha i} \right).
\]
Recall the basic fact that the generating function $H(t)$ of the homogeneous complete symmetric polynomials $h_i$ can be written as follows:
\[
	H(t) = \sum_{i = 0} h_i t^i =  \prod_{i} \frac{1}{1-x_it}  = \prod_{i \geq 1} \exp \left( \frac{p_i}{i}t^i \right).
\]
This implies that 
\[
	\sum_\lambda  \frac{ \eta^{\lambda}_{\alpha} J_\lambda}{H^*_\lambda H_*^\lambda} = \prod_{i \geq 2} \exp \left( \frac{p_i}{\alpha i} \right)	= e^{-h_1/\alpha } \prod_i (1-x_i)^{-1/\alpha}.
\]
Following Stanley~\cite{Stanley89}, we have $\sum_k J_{k}/(\alpha^kk!) =  \prod_i (1-x_i)^{-1/\alpha}$. This, along with the fact that $h_1 = J_1$ gives us
\[
		\sum_\lambda  \frac{ \eta^{\lambda}_{\alpha} J_\lambda}{H^*_\lambda H_*^\lambda} = e^{-J_1/\alpha } \sum_{k} \frac{J_k}{\alpha^k k! } = \sum_{j,k} \frac{(-1)^j}{\alpha^jj! \alpha^kk!} J_1^j J_k.
\]
The Pieri rule for Jack polynomials~\cite{Stanley89} implies that
\[
 	J_1^j J_k =  \sum_{\lambda = (k) + j} d^{\lambda/k}  \left( \frac{H^{(k)}_*}{H^\lambda_*} \right) J_\lambda.
\]
where the sum ranges over all shapes $\lambda$ obtained by adding $j$ inner corners to $(k)$ in succession. Equating coefficients and then reindexing gives
\[
	\frac{ \eta^{\lambda}_{\alpha}}{H^*_\lambda H_*^\lambda} = \sum_{\lambda = (k) + j} \frac{(-1)^j}{\alpha^jj! \alpha^kk!}  	  \left( \frac{H^{(k)}_*}{H^\lambda_*} \right) d^{\lambda/k} = \sum_{k=0}^{|\lambda|} \frac{(-1)^{|\lambda|-k}}{\alpha^{|\lambda|-k} (|\lambda|-k)! \alpha^kk!}    \left( \frac{H^{(k)}_*}{H^\lambda_*} \right) d^{\lambda/k}.
\]
Let $d^\lambda = |\lambda|!/H_\lambda^*$. By~\cite[Proposition 5.2]{OkounkovO97}, we have
\[
	d^{\lambda / \mu} = \frac{d^\lambda P_\mu^\star(\lambda)}{|\lambda| (|\lambda| -1) \cdots (|\lambda|-|\mu|+1) }.
\]
Multiplying both sides by $H^*_\lambda H_*^\lambda$ and applying \cite[Proposition 5.2]{OkounkovO97} gives us
\begin{align*}
 \eta^{\lambda}_{\alpha} &= \sum_{k=0}^{|\lambda|} \frac{(-1)^{|\lambda|-k} H^*_\lambda H^{(k)}_*}{\alpha^{|\lambda|-k} (|\lambda|-k)! \alpha^kk!}   \left( \frac{d^{\lambda} P_k^\star(\lambda)}{ |\lambda| (|\lambda| -1) \cdots (|\lambda|-k+1)} \right)\\
 &=  \sum_{k=0}^{|\lambda|} \frac{(-1)^{|\lambda|-k} \alpha^{|\lambda|} |\lambda|! H^{(k)}_*}{\alpha^{|\lambda|-k} (|\lambda|-k)! \alpha^kk!}  \left(  \frac{P_k^\star(\lambda)}{ |\lambda| (|\lambda| -1) \cdots (|\lambda|-k+1)} \right)\\
 &=  \sum_{k=0}^{|\lambda|}\frac{(-1)^{|\lambda|-k}}{k!} H_*^{(k)}    P_k^\star(\lambda).
\end{align*}
By definition, we have $ H_*^{(k)}  P_k^\star(\lambda) = J_k^\star(\lambda)$, which completes the proof.
\end{proof}

\section{Tableau Transversals and Principal Hook Products}\label{sec:php}\label{sec:0}

We now leverage some combinatorial results of~\cite{AlexanderssonF17,FilmusL22a} to give a more tractable combinatorial formulation of Theorem~\ref{thm:jack}, which we use to prove Theorem~\ref{thm:main} for $\alpha = 0,1,2$.

A \emph{$k$-transversal} $T$ of $\lambda$ is a set of $k$ cells of $T$ which forms a partial transversal of the columns of $\lambda$, that is, no two cells of $T$ lie in the same column of $\lambda$. Define the \emph{$\alpha$-weight} of a $k$-transversal $T$ to be the lower hook product of $T$, i.e., $w_\alpha (T) = H^T_*$, with the convention that $w_\alpha(\emptyset) = 1$ (see Figure~\ref{fig:transversals} for examples). Let $\mathcal{T}_\lambda^k$ be the collection of $k$-transversals of $\lambda$. 

In~\cite[Theorem 5.12]{AlexanderssonF17}, Alexandersson and F\'eray show that
\begin{align}\label{eq:af}
	\frac{J^\star_{k}(\lambda)}{k!} = \sum_{T \in \mathcal{T}_\lambda^k} w_\alpha(T).
\end{align}
Independently, Filmus and Lindzey~\cite{FilmusL22a} prove the following combinatorial identity
\begin{align}\label{eq:fl}
\frac{J_{\lambda_1}^\star(\lambda)}{\lambda_1!} = \sum_{T \in \mathcal{T}_\lambda^{\lambda_1}} w_\alpha(T)  = H^1_*(\lambda).
\end{align}
\noindent For $\alpha = 1$, we note that Equation~(\ref{eq:fl}) can also be observed from Naruse's hook-length formula for standard skew-tableaux~\cite{Naruse14}. We write $\mu \preceq_k \lambda$ if $\mu$ is a subshape $\lambda$ obtained by removing $k$ columns of $\lambda$. There are $\binom{\lambda_1}{k}$ such subshapes, and we let the sigma notation $\sum_{\mu \preceq_k \lambda}$ denote the sum over all $\binom{\lambda_1}{k}$ subshapes $\mu$ of $\lambda$ obtained by removing $k$ columns.

\begin{figure}
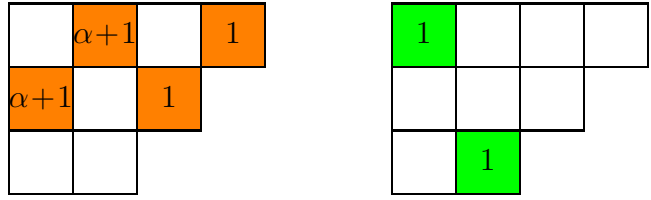

\ytableausetup{mathmode, boxsize=2em}
\[ 
	\begin{ytableau}
	 ~ & *(orange)\alpha\!+\!1  & ~ & *(orange)1 \\ 
	*(orange)\alpha\!+\!1 &  ~ & *(orange)1  \\
	~ & ~ 
	\end{ytableau}
 \quad \quad \quad \quad 
	\begin{ytableau}
	*(green) 1 & ~  & ~ & ~  \\ 
	~ &  ~ & ~  \\
	~ & *(green) 1 
	\end{ytableau}
\]
\caption{Let $\mu = (4,3,2) \vdash 9$. The colored cells $S = \{(2,1), (1,2),(2,3),(1,4)\}$ on the left is a $4$-transversal of $\mu$ with $\alpha$-weight $w_\alpha(S) = (\alpha + 1)^2$. The colored cells $S' = \{(1,1),(3,2)\}$ on the right is a $2$-transversal of $\mu$ with $\alpha$-weight $w_\alpha(S') = 1$. Each colored cell is labeled with its lower hook length with respect to $S$ and $S'$. }\label{fig:transversals}
\end{figure}

\begin{theorem}\label{thm:php}\label{cor:eigs}  For any shape $\lambda$ and $\alpha \in \mathbb{R}$, we have 
\[
	\eta^\lambda_\alpha = (-1)^{|\lambda|-\lambda_1}\sum_{k=0}^{\lambda_1} (-1)^{k} \sum_{ \mu \preceq_{k} \lambda } H^1_*(\mu).
\]
\end{theorem}
\begin{proof}
We claim that 
\[
	\frac{J^\star_{k}(\lambda)}{k!} = \sum_{T \in \mathcal{T}_\lambda^k} w_\alpha(T) =  \sum_{ \mu \preceq_{\lambda_1-k} \lambda } H^1_*(\mu).
\]
The first equality follows from Equation~(\ref{eq:af}), and the second equality follows from applying Equation~(\ref{eq:fl}) to each shape $\mu$ obtained by removing $\lambda_1 - k$ columns from $\lambda$, so that $\mu_1 = k$. Note for all $\ell > \lambda_1$ that there exists no $\mu$ such that $\mu \preceq_\ell \lambda$. Reindexing Theorem~\ref{thm:jack} and applying the identity above gives
\[
\eta^\lambda_\alpha =  \sum_{k=0}^{|\lambda|} \frac{(-1)^{|\lambda|-k}}{k!}   J_k^\star(\lambda)\\
	=  \sum_{k=0}^{\lambda_1} \frac{(-1)^{|\lambda|-k}}{k!}   J_k^\star(\lambda)\\
	= (-1)^{|\lambda|-\lambda_1}\sum_{k=0}^{\lambda_1} (-1)^{k} \sum_{ \mu \preceq_{k} \lambda } H^1_*(\mu),
\]
as desired.
\end{proof}
\noindent We are now ready to give an elementary combinatorial proof of Theorem~\ref{thm:main} for $\alpha = 1,2$ using the Principle of Inclusion-Exclusion. This is due to the fact that $\lambda$-colored permutations $\mathcal{S}_\lambda$ (defined in Section~\ref{sec:intro}) and \emph{$\lambda$-colored perfect matchings} $\mathcal{M}_\lambda$ (defined below) are \emph{bona fide} combinatorial objects, and their cardinalities are counted by principal lower hook products.

\begin{theorem}\emph{\cite{FilmusL22a}} For any shape $\lambda$, we have
$$|\mathcal{S}_\lambda|, |\mathcal{M}_\lambda| = H_*^1(\lambda)$$
for $\alpha = 1,2$, respectively.
\end{theorem}

For each $i \in [2n] := \{1,2,\ldots,n\}$, we assign a list of colors $L(i)$ such that $L(i) = L(i+1)$ for all odd $i$. A \emph{colored perfect matching} $(c,m)$ is an assignment of colors $c = c_1,c_2,\ldots, c_n$ such that $c_i \in L(i)$ and a perfect matching $m \in \mathcal{M}_{2n}$ such that $m(i) = j \Rightarrow c_i = c_j$, where $m(i)$ denotes the partner of $i$ in the perfect matching $m$. Any partition $\lambda$ defines a color list on each element $i$ of the symbol set $[2\lambda_1]$ by setting $L(i) = L(i+1) = [\lambda'_i]$. Let $\mathcal{M}_\lambda$ to be the set of all such colored perfect matchings, formally, 
\[
	\mathcal{M}_\lambda := \{ (c \in [\lambda'_1] \times \cdots \times [\lambda'_{\lambda_1}], m \in \mathcal{M}_{2\lambda_1}) : m(i) = j \Rightarrow c_i = c_j \text{ for all } i \in [2\lambda_1]\}.
\]
We say that a colored perfect matching $(c,m) \in \mathcal{M}_\lambda$ is a \emph{derangement} if $m(i) = i+1 \Rightarrow c_i \neq 1$ for all odd $1 \leq i < 2\lambda_1$. These are the colored perfect matchings that have no edges in common with $(1,\ldots,1, \{\{1,2\},\ldots,\{2\lambda_1-1,2\lambda_1\} \}) \in \mathcal{M}_\lambda$. Let $\mathcal{D}_\lambda'$ be the set of derangements of $\mathcal{M}_\lambda$.

\begin{proof}[Proof of Theorem~\ref{thm:main} for $\alpha = 1,2$]
Let $\alpha = 1$. A \emph{fixed point} of a colored permutation is a symbol $i$ such that $c(i) = 1$ and $\sigma(i) = i$. Consider the summation $\sum_{ \mu \preceq_{k} \lambda } H^1_*(\mu)$ which ranges over each shape $\mu$ obtained by removing $k$ columns from $\lambda$. For each $\mu$ in this summation, the indices $I \subseteq [\lambda_1]$ of the $k$ columns removed from $\lambda$ to obtain $\mu$ correspond to $k$ fixed points of a $\lambda$-colored permutation, and the number of colored permutations on the remaining columns is counted by $H^1_*(\mu)$. Thus it counts the number of $\lambda$-colored permutations that have each $i \in I$ as a fixed point. This overcounts the number of $\lambda$-colored permutations for which $I$ is the exactly the set of fixed points, so we must exclude those $\lambda$-colored permutations for which $I$ is a proper subset of its set of fixed points. Thus by the Principle of Inclusion-Exclusion, the alternating sum in Theorem~\ref{thm:php} for $\alpha = 1$ counts the number of $\lambda$-colored permutations with exactly 0 fixed points, as desired.

The proof for $\alpha = 2$ is identical \emph{mutatis mutandis} and shows $\eta_2^\lambda = (-1)^{|\lambda|-\lambda_1} |\mathcal{D}_\lambda'| = (-1)^{|\lambda|-\lambda_1} D^\lambda_2$, where the last equality is a combinatorial exercise left to the reader.
\end{proof}
\medskip

\noindent In Section~\ref{sec:main} we give a generalization of the proof above to all $\alpha \in \mathbb{R}$, but along the way we collect several results concerning principal lower hook products, perhaps of independent interest, that allow us to give a closed-form expression of Theorem~\ref{thm:main}. Moreover, the specialization to $\alpha = 1,2$ leads to nice expressions for the eigenvalues of the derangement graphs. The reader uninterested in such closed-form expressions may skip to Section~\ref{sec:main}.

For didactical reasons, we conclude this section with a proof of the first main result for $\alpha = 0$, as it is simple and  provides some insight into the general $\alpha \in \mathbb{R}$ case.  
\begin{theorem}\label{thm:0}
	For all $\lambda$, we have
	\[
		\eta^\lambda_0 = (-1)^{|\lambda|-\lambda_1} \prod_{i=1}^{\ell(\lambda')} (\lambda_i'-1).
	\]
\end{theorem}
\begin{proof}
	Let $x = x_1,\ldots, x_n$ be the roots of a polynomial $p(z)$. Recall \emph{Vieta's formula} 
	\[
		p(z) = \prod_{i=1}^n (z-x_i) = \sum_{k=0}^n (-1)^{k} e_{k}(x_1,\cdots,x_n)z^{n-k}. 
	\]
	By Theorem~\ref{thm:jack}, we have  
	\[
		\eta^\lambda_0 =  \sum_{k=0}^{|\lambda|} (-1)^{|\lambda|-k}   J_k^\star(\lambda)/k!. 
	\]
	For $\alpha = 0$, Theorem~\ref{thm:php} implies that $J_k^\star(\lambda)/k! = |\mathcal{T}_\lambda^k|$, which gives us
	\[
		\eta^\lambda_0 =  \sum_{k=0}^{|\lambda|} (-1)^{|\lambda|-k}   J_k^\star(\lambda)/k! = (-1)^{|\lambda|} \sum_{k=0}^{|\lambda|} (-1)^{k} |\mathcal{T}_\lambda^k| = (-1)^{|\lambda|} \sum_{k=0}^{|\lambda|} (-1)^{k} e_{k}(\lambda').
	\]
	Setting $z = 1$ and $x = \lambda'$ in Vieta's formula gives us
	\[
		\eta^\lambda_0 = (-1)^{|\lambda|} \prod_{i=1}^{\lambda_1} (1-\lambda_i') = (-1)^{|\lambda|-\lambda_1} \prod_{i=1}^{\ell(\lambda')} (\lambda_i'-1),
	\]
	as desired.
\end{proof}
\noindent To see that this proves Theorem~\ref{thm:main} for $\alpha = 0$, first note that the effect of setting $\alpha = 0$ is that the arm lengths of cells in $T \in \mathcal{T}^k_\lambda$ are ignored, thus we associate the identity permutation $()$ to each $T$. Let 
\[
 \mathcal{D}^\lambda_0 = \{ (c \in [\lambda'_1] \times \cdots \times [\lambda'_{m}], ()) : c(i) \neq 1 \text{ for all } i \} \subseteq \mathcal{D}^\lambda
\]
be the derangements that move no symbols of $[\lambda_1]$. Clearly, $|\mathcal{D}^\lambda_0| = \prod_{i=1}^{\ell(\lambda')} (\lambda_i'-1)$, as desired. Evidently, we may define the Jack derangements at $\alpha = 0$ to be the words of $[\lambda'_1] \times \cdots \times [\lambda'_{m}]$ that avoid the symbol 1. 
For $\alpha \neq 0$ we must take into account the arm lengths of the cells, which requires a more detailed examination of the principal lower hook product.

\section{Minors of the Principal Hook Product}\label{sec:minors}

In this section we prove a few technical lemmas concerning the principal hook product that are needed for closed-form expressions of Theorem~\ref{thm:main}. 
Let $\lambda^{-i}$ be the shape obtained by removing the $i$th column of $\lambda$. Let $\lambda^{-i_1 -i_2-\cdots-i_k}$ be the shape obtained by removing (distinct) columns $i_1,i_2,\ldots,i_k$ of $\lambda$. It is useful to think of the $H^1_*(\lambda^{-i})$'s as the \emph{first minors} of $\lambda$, and the $H^1_*(\lambda^{-i_1-\cdots-i_k})$'s as \emph{k-minors} of $\lambda$. The ordering of the $i_j$'s is immaterial, i.e.,
$$\lambda^{-i_1 -i_2-\cdots-i_k} = \lambda^{-i_{\sigma(1)} -i_{\sigma(2)}-\cdots-i_{\sigma(k)}} \quad  \text{ for all $\sigma \in S_k$}
.$$ 
Let $\lambda^{\underline{k}}$ be the shape obtained by removing the last $k$ columns of $\lambda$. We adopt the shorthand $h_j := h^\lambda_*(1,j)$ henceforth.
Lemma~\ref{prop:h1} gives a Laplace-like expansion that relates the principal lower hook product to its first minors.
\begin{lemma}[Laplace Expansion]\label{prop:h1} For all $\lambda$, we have 
\[
	\sum_{i = 1}^{\lambda_1} H^1_*(\lambda^{-i}) = \frac{1}{\alpha} \left( H^1_*(\lambda) + (\alpha - h_{\lambda_1}) 		H^1_*(\lambda^{\underline{1}}) \right), \text{ equivalently,}
\]
\[
	H^1_*(\lambda)  = \sum_{i = 1}^{\lambda_1-1} \alpha H^1_*(\lambda^{-i}) + h_{\lambda_1} 		H^1_*(\lambda^{-\lambda_1}).
\]
\end{lemma}
\begin{proof}
Let $h_j' := h_j - \alpha$. We can write any first minor of $\lambda$ in terms of hook products of $\lambda$:
$$ H^1_*(\lambda^{-i}) = \prod_{j < i} h'_j  \prod_{j > i} h_j,$$
which implies that
\begin{align*}
\sum_{i = 1}^{\lambda_1} \frac{H^1_*(\lambda^{-i})}{H^1_*(\lambda)}  = \sum_{i = 1}^{\lambda_1} \frac{1}{h_i} \prod_{j < i} \frac{h'_j }{h_j} = \frac{1}{h_1} + \cdots +   \left[ \prod_{j=1}^{\lambda_1-2} \frac{h_j'}{h_j} \right] \frac{1}{h_{\lambda_1-1}} + \left[ \prod_{j=1}^{\lambda_1-2} \frac{h_j'}{h_j}  \right] \frac{h'_{\lambda_1-1}}{h_{\lambda_1-1}} \frac{1}{h_{\lambda_1}}.
\end{align*}
~\\
Note that this sum telescopes to 1 if and only if $\alpha = 1$ and $h_{\lambda_1} = 1$. In general, we have 
\begin{align*}
\sum_{i = 1}^{\lambda_1} \frac{H^1_*(\lambda^{-i})}{H^1_*(\lambda)}  
	&= \frac{1}{h_1} + \cdots +   \left[ \prod_{j=1}^{\lambda_1-2} \frac{h_j'}{h_j} \right] \frac{1}{h_{\lambda_1-1}} + \left[ \prod_{j=1}^{\lambda_1-1} \frac{h_j'}{h_j}  \right] \frac{1}{h_{\lambda_1}}\\
	&= 1 - \frac{(h_{\lambda_1}-1)}{h_{\lambda_1}}\prod_{j=1}^{\lambda_1 - 1}\frac{h'_j}{h_j}  -   (\alpha - 1)\sum_{i=1}^{\lambda_1 - 1} \frac{h_1'}{h_1} \cdots \frac{h'_{\lambda_1 - i - 1}}{h_{\lambda_1 - i - 1}} \cdot \frac{1}{h_{\lambda_1 - i }}.
\end{align*}
Multiplying both sides by $H^1_*(\lambda)$ gives 
\begin{align}\label{eq:minors}
\sum_{i = 1}^{\lambda_1} H^1_*(\lambda^{-i}) = H^1_*(\lambda) - (h_{\lambda_1}-1) H^1_*(\lambda^{-\lambda_1})  -   (\alpha - 1)\sum_{i=1}^{\lambda_1 - 1} H^1_*(\lambda^{-i}).
\end{align}
After rearranging terms and noting that $\lambda^{-\lambda_1} = \lambda^{\underline{ 1}}$, we have
\[
\sum_{i = 1}^{\lambda_1} H^1_*(\lambda^{-i}) = \frac{1}{\alpha}(H^1_*(\lambda) + (\alpha - h_{\lambda_1}) H^1_*(\lambda^{\underline{1}})),
\]
as desired. Rearranging once more finishes the proof.
\end{proof}

For $\alpha \geq 1$, we are now in a position to give a short proof of both the Alternating Sign Theorem and a useful upper bound on the magnitudes of the Jack derangement sums. 
\begin{proposition}\label{prop:upperbound} For all $\alpha \geq 1$, we have $\emph{sgn}~\eta_\alpha^\lambda = (-1)^{|\lambda|-\lambda_1}$. Moreover, $|\eta_\alpha^\lambda| \leq H^1_*(\lambda)$.
\end{proposition}
\begin{proof}
Since $\alpha, h_{\lambda_1} \geq 1$, applying Equation~(\ref{eq:minors}) repeatedly shows that  
\begin{align}\label{eq:chain}
\sum_{\mu \preceq_{\lambda_1} \lambda} H_1^*(\mu)  \leq \cdots  \leq \sum_{\mu \preceq_1 \lambda} H_1^*(\mu)  \leq  H_1^*(\lambda). 
\end{align}
If $|\lambda|-\lambda_1$ is even, then by Corollary~\ref{cor:eigs} we have
$$ 0 \leq  H_1^*(\lambda) - \sum_{\mu \preceq_1 \lambda} H_1^*(\mu) \leq \eta_\alpha^\lambda; \quad \text{ otherwise, } \quad 
 0 \geq  -H_1^*(\lambda) + \sum_{\mu \preceq_1 \lambda} H_1^*(\mu) \geq \eta_\alpha^\lambda,$$
i.e., $\text{sgn}~\eta_\alpha^\lambda = (-1)^{|\lambda|-\lambda_1}$. That $|\eta_\alpha^\lambda| \leq H^1_*(\lambda)$ follows from Equation~(\ref{eq:chain}) and Theorem~\ref{thm:php}.
\end{proof}

For any $\lambda$ and integer $0 \leq j \leq \lambda_1-1$, let
\[
	f^*_\lambda(j) :=  \prod_{i=0}^j ((j+1)\alpha - h_{\lambda_1-i}),
\]
and define $f^*_\lambda(j) := 1$ for all negative integers $j$. For the proof of the next lemma, it will be useful to define the following related quantity: 
\[
	f^*_\lambda(j,i) :=  \prod_{l=0}^{j-i} ((j+1)\alpha - h_{\lambda_1-l}) \prod_{l=j-i+1}^j ((j+2)\alpha - h_{\lambda_1-l-1}).
\]
In other words, $f^*_\lambda(j,i)$ is the function obtained by both incrementing the $\alpha$-coefficient by 1 and decrementing the hook index by 1 in the last $i$ factors $f^*_\lambda(j)$. Lemma~\ref{lem:split} is a generalization of Lemma~\ref{prop:h1} that will lead to a more explicit version of~\cite[Theorem 5.12]{AlexanderssonF17}.
\begin{lemma}\label{lem:split} For all shapes $\lambda$ and $0 \leq j \leq \lambda_1-1$, we have
\[
	\sum_{i = 1}^{\lambda_1} f^*_{\lambda^{-i}}(j-1)~H^1_*( (\lambda^{-i})^{\underline{j}}) = \frac{1}{\alpha}\left(f^*_{\lambda}(j-1)~H^1_* (\lambda^{\underline{j}}) +  f^*_{\lambda}(j) ~H^1_*(\lambda^{\underline{j+1}})\right).\\
\]
\end{lemma}
\begin{proof}
We begin by listing a few combinatorial facts that are easily verified.
\begin{enumerate}
\item For all $j$, we have $h_{(\lambda^{-i})_1-j} = h_{\lambda_1-j} \text{ if } i \leq \lambda_1-j; \text{ otherwise, } h_{(\lambda^{-i})_1-j} = h_{\lambda_1 - 1 - j} - \alpha$.
\item For all $i \leq \lambda_1 - j$, we have $(\lambda^{-i})^{\underline{j}} = (\lambda^{\underline{j}})^{-i}$.
\item For all $i > \lambda_1 - j$ we have $(\lambda^{-i})^{\underline{j}} = \lambda^{\underline{ j+1}}$.
\item For all $j$, we have $h_{(\lambda^{\underline{j}})_1} = h_{\lambda_1- j} - j\alpha$.
\end{enumerate}
The first three facts allows us to split the summation as follows:
	\begin{align*}
		&\sum_{i = 1}^{\lambda_1} f^*_{\lambda^{-i}}(j-1) ~ H^1_*( (\lambda^{-i})^{\underline{j}}) = f^*_{\lambda}(j-1) \sum_{i = 1}^{\lambda_1-j}  H^1_*( (\lambda^{\underline{j}})^{-i}) + \sum_{i=1}^j f^*_{\lambda}(j-1,i) ~H^1_*(\lambda^{\underline{ j+1}}).
		\intertext{By Lemma~\ref{prop:h1} and the last fact, we have}
		&= \frac{1}{\alpha} \left(f^*_{\lambda}(j-1) \left(H^1_*(\lambda^{\underline{j}}) + (\alpha - h_{\lambda^{\underline{j}}_1}) H^1_*(\lambda^{\underline{ j+1}})\right)  + \alpha \sum_{i=1}^j f^*_{\lambda}(j-1,i) H^1_*(\lambda^{\underline{ j+1}})\right)\\
		&= \frac{1}{\alpha} \left(f^*_{\lambda}(j-1) \left(H^1_*(\lambda^{\underline{ j}}) + ((j+1)\alpha - h_{\lambda_1- j}) H^1_*(\lambda^{\underline{ j+1}})\right) + \alpha \sum_{i=1}^j f^*_{\lambda}(j-1,i) H^1_*(\lambda^{\underline{ j+1}})\right)\\
		&= \frac{1}{\alpha}  \left(f^*_{\lambda}(j-1) H^1_*(\lambda^{\underline{ j}}) + \left[  ((j+1)\alpha - h_{\lambda_1- j})f^*_{\lambda}(j-1,0) + \alpha \sum_{i=1}^j f^*_{\lambda}(j-1,i) \right] H^1_*(\lambda^{\underline{ j+1}}) \right).
	\end{align*}
	It suffices to show that the bracketed factor equals $f^*_\lambda(j)$. We may write the summation as
	\begin{align*}
		\sum_{i=1}^j f^*_\lambda(j-1,i) = (j\alpha - h_{\lambda_1}) \cdots (j\alpha - h_{\lambda_1-j+2}) ~&\cdot~ ((j+1)\alpha - h_{\lambda_1-j}) ~ + \\
		&\vdots\\ 
		\\
		(j\alpha - h_{\lambda_1})(j\alpha - h_{\lambda_1-1}) ~ \cdot ~ & ((j+1)\alpha - h_{\lambda_1-3}) ~\cdots ~((j+1)\alpha - h_{\lambda_1-j})~+\\
		(j\alpha - h_{\lambda_1}) ~ \cdot ~ & ((j+1)\alpha - h_{\lambda_1-2}) ~\cdots ~((j+1)\alpha - h_{\lambda_1-j})~+\\
		&((j+1)\alpha - h_{\lambda_1-1}) ~\cdots ~((j+1)\alpha -  h_{\lambda_1-j}).
	\end{align*}
		which we may write as
	\begin{align*}
			\sum_{i=1}^j f^*_\lambda(j-1,i) &= ((j+1)\alpha - h_{\lambda_1-j})~[~ (j\alpha - h_{\lambda_1}) \cdots (j\alpha - h_{\lambda_1-j+2}) \\
			&+ ((j+1)\alpha - h_{\lambda_1-(j-1)})~[~ (j\alpha - h_{\lambda_1}) \cdots (j\alpha - h_{\lambda_1-j+3}) \\
			&+ ((j+1)\alpha - h_{\lambda_1-(j-2)})~[~ (j\alpha - h_{\lambda_1}) \cdots (j\alpha - h_{\lambda_1-j+4}) \\
			&~\vdots\\ 
		&+ ((j+1)\alpha - h_{\lambda_1-1})~] \cdots ].
	\end{align*}
		We may factor out $((j+1)\alpha - h_{\lambda_1-j})$, leaving
	\begin{align*}
		(j\alpha - h_{\lambda_1}) \cdots (j\alpha - h_{\lambda_1-j+2}) &+ ((j+1)\alpha - h_{\lambda_1-j+1})~[~ (j\alpha - h_{\lambda_1}) \cdots (j\alpha - h_{\lambda_1-j+3}) \\
		&+ ((j+1)\alpha - h_{\lambda_1-j+2})~[~ (j\alpha - h_{\lambda_1}) \cdots (j\alpha - h_{\lambda_1-j+4}) \\
		&~\vdots\\ 
		&+ ((j+1)\alpha - h_{\lambda_1-1})~] \cdots ].
	\end{align*}
		We have $$\alpha (j\alpha - h_{\lambda_1}) \cdots (j\alpha - h_{\lambda_1-j+2}) + f_\lambda^*(j-1,0) = (j\alpha - h_{\lambda_1}) \cdots (j\alpha - h_{\lambda_1-j+2})((j+1)\alpha - h_{\lambda_1-j+1}),$$ so we may factor out $((j+1)\alpha - h_{\lambda_1-j+1})$. Continuing in this manner gives us
		\[
		= \frac{1}{\alpha} \left(f^*_{\lambda}(j-1)H^1_*(\lambda^{\underline{ j}}) + f^*_{\lambda}(j) H^1_*(\lambda^{\underline{j+1}}) \right),\quad\quad\quad\quad\quad\quad\quad\quad\quad\quad\quad\quad\quad\quad\quad\quad\quad\quad\quad\quad
		\]
	as desired.
\end{proof}

\noindent Theorem~\ref{thm:k} is a more explicit form for \cite[Theorem 5.12]{AlexanderssonF17}, perhaps of independent interest.

\begin{theorem}\label{thm:k} For all $\lambda$ and $\alpha \in \mathbb{R}$, we have 
	\[
		\frac{J^\star_{\lambda_1 - k}(\lambda)}{(\lambda_1 - k)!} = \sum_{\mu \preceq_k \lambda} H^1_*(\mu) = \frac{1}{\alpha^k} \sum_{j =0}^k (-1)^j \frac{ \prod_{i=1}^{\lambda_1} (h_i - j\alpha)}{{ (k-j)! j!}}, \text{ equivalently,}
	\]
	\[
			\frac{H^*_k}{(\lambda_1 - k)!} J^\star_{\lambda_1 - k}(\lambda) = \sum_{j =0}^k (-1)^j \binom{k}{j} \prod_{i=1}^{\lambda_1} (h_i - j\alpha).
	\]
\end{theorem}
\begin{proof}
	First, note that
	\[
	k! \sum_{\mu \preceq_k \lambda} H^1_*(\mu) = \!\!\!\! \sum_{\substack{\text{distinct } i_1, \cdots, i_k \\ i_j \in [\lambda_1]}} \!\!\!\!\!\! H^1_*(\lambda^{-i_1-\cdots-i_k}) = \sum_{\mu \preceq_1 \lambda} ~ \sum_{\nu^1 \preceq_1 \mu} \cdots \!\!\!\!\!\! \sum_{\nu^{k-1} \preceq_1 \nu^{k-2}} \!\!\!\! H^1_*(\nu^{k-1}),
	\]
	so it suffices to prove
	\[
		\sum_{\mu \preceq_1 \lambda} ~ \sum_{\nu^1 \preceq_1 \mu} \cdots \!\!\!\!\!\! \sum_{\nu^{k-1} \preceq_1 \nu^{k-2}} \!\!\!\! H^1_*(\nu^{k-1}) = \frac{1}{\alpha^k} \sum_{j =0}^k \binom{k}{j} f_\lambda^*(j-1) H^1_*(\lambda^{\underline{j}}).
	\]
	We proceed by induction on $k$. Applying the induction hypothesis to each $\lambda^{-i} \preceq_1 \lambda$ gives	
	\[
		k! \sum_{\mu \preceq_k \lambda} H^1_*(\mu) = \frac{1}{\alpha^{k-1}}  \sum_{i=1}^{\lambda_1} \sum_{j = 0}^{k-1}\binom{k-1}{j}~f_{\lambda^{-i}}^*(j-1)~H^1_*( (\lambda^{-i})^{\underline{j}}) .
	\]
	After interchanging sums and then applying Lemma~\ref{lem:split} to each inner sum, we have
	\[
		k! \sum_{\mu \preceq_k \lambda} H^1_*(\mu) = \frac{1}{\alpha^{k}}  \sum_{j = 0}^{k-1} \binom{k-1}{j}  
		 \left[f^*_{\lambda}(j-1)H^1_*(\lambda^{\underline{ j}}) + f^*_{\lambda}(j) H^1_*(\lambda^{\underline{j+1}}) \right].
	\]
	Pascal's formula implies that
	\[
		k! \sum_{\mu \preceq_k \lambda} H^1_*(\mu) = \frac{1}{\alpha^{k}} \sum_{j=0}^{k} \binom{k}{j} ~f_\lambda^*(j-1) H^1_*( \lambda^{\underline{j}}).
	\]
	Finally, we have
	\[
		f_\lambda^*(j-1) H^1_*( \lambda^{\underline{j}}) = \prod_{i=0}^{j-1} (j\alpha  - h_{\lambda_1-i}) \prod_{i=1}^{\lambda_1 - j} (h_i - j\alpha) = (-1)^j \prod_{i=1}^{\lambda_1} (h_i - j\alpha),
	\]
	which completes the proof.
\end{proof}
\noindent Those familiar with the umbral calculus or the calculus of finite differences may recognize the right-hand side of the second equation in Theorem~\ref{thm:k} as essentially the $k$th-order forward difference $\Delta^k$ of the univariate degree-$\lambda_1$ polynomial 
\[
	\mathbf{H}^1_*(\lambda,x) := \prod_{i=1}^{\lambda_1} (h_i - x\alpha)
\]
in $x$ at the origin, i.e.,
\begin{align}
	\frac{H^*_{(k)}}{(\lambda_1 - k)!} J^\star_{\lambda_1 - k}(\lambda) = (-1)^k \Delta^k[\mathbf{H}^1_*(\lambda,x)](0)
\end{align}
where $\Delta^k[f](x) := \sum^k_{i=0} (-1)^{k-i} \binom{k}{i} f(x + i)$ for any function $f(x)$. Forward differences of this kind are connected to polynomial interpolation in the falling factorial basis 
$$x^{\underline{k}} := x(x-1)(x-2)\cdots(x-k+1),$$ 
in particular, the \emph{Newton (interpolation) polynomial} $N(x)$ of a set of points $S = \{(x_i,p(x_i))\}_{i=0}^d$:
\[
 N(x) := [p(x_0)] x^{\underline{0}} + [p(x_0),p(x_1)] x^{\underline{1}} + \cdots + [p(x_0),p(x_1),\ldots, p(x_d)] x^{\underline{d}}
\]
where $[p(x_0),\ldots, p(x_j)]$ is the notation for the so-called \emph{$j$th divided difference}. Note that if $p(x)$ is a degree-$d$ polynomial and $|S| > d+1$, then $[p(x_0),\ldots, p(x_{j})] = 0$ for all $j > d$.
 
Finally, we recall the well-known fact that if $x_i = i$ for all $0 \leq i \leq d$, then 
\[
	[p(x_0),p(x_1),\ldots, p(x_j)] = \frac{\Delta^j[p](0)}{j!},
\]
and the Newton interpolation polynomial is of the form 
\begin{align}
 	N(x) = \frac{p(0)}{0!} x^{\underline{0}} + \frac{\Delta^1 [p](0)}{1!} x^{\underline{1}} + \cdots + \frac{\Delta^d [p](0)}{d!} x^{\underline{d}}.
\end{align}
See Stanley~\cite[Ch.~1.9]{Stanley2011} for a more in-depth discussion of the calculus of finite differences and its connections to combinatorics. In the next section, we show that each Jack derangement number is the sum of the coefficients of a Newton polynomial. 

\section{Proof of Theorem~\ref{thm:main}}\label{sec:main}
Building off the results of the previous sections, we give a proof of Theorem~\ref{thm:main} in this section. For all $j > 0$, define
\[
	H^1_*(\lambda,j) := \prod_{i=1}^{\lambda_1} (h_i - j\alpha)
\]
to be the \emph{$j$-shifted principal lower hook product}. It will be convenient to think of the shifted principal lower hook product as a univariate polynomial in $x$, i.e.,
\[
	\mathbf{H}^1_*(\lambda,x) := \prod_{i=1}^{\lambda_1} (h_i - x\alpha).
\]
We let $d^{(\alpha)}_{n,k}$ denote the $\alpha$-generalization of the \emph{rencontres numbers}, that is, 
\[
	d^{(\alpha)}_{n,k} := \frac{\alpha^{n}n!}{\alpha^k k!}\sum_{i=0}^{n-k} \frac{(-1)^{i}}{\alpha^i i!}.
\]
For $\alpha = 1$, the rencontres numbers $d_{n,k} := d^{(1)}_{n,k}$ count the number of permutations of $S_n$ that have precisely $k$ fixed points.
\begin{theorem}\label{thm:dnk}
For all $\lambda$, $\alpha \in \mathbb{R}$, and $n \geq \lambda_1$, we have 
\[
	\eta_\alpha^\lambda = (-1)^{|\lambda|-\lambda_1}  \frac{1}{\alpha^{n} n!} \sum_{j = 0}^{n} d^{(\alpha)}_{n,j} H^1_*(\lambda,j).
\]
\end{theorem}
\begin{proof}
By Theorem~\ref{cor:eigs} we have
\begin{align*}
	\eta^\lambda_\alpha &= (-1)^{|\lambda|-\lambda_1} \sum_{k=0}^{\lambda_1} (-1)^k \sum_{\mu \preceq_k \lambda} H^1_*(\mu).\\
	\intertext{By Theorem~\ref{thm:k}, we have}
	&= (-1)^{|\lambda|-\lambda_1} \sum_{k=0}^{\lambda_1} \frac{ (-1)^k }{\alpha^k} \sum_{j =0}^k (-1)^j \frac{ H^1_*(\lambda,j) }{{ (k-j)! j!}}.\\
	\intertext{Interchanging summations gives us}
	& = (-1)^{|\lambda|-\lambda_1}  \sum_{j = 0}^{\lambda_1} \sum_{k=j}^{\lambda_1} \frac{(-1)^{k-j}}{\alpha^k}  \frac{H^1_*(\lambda,j)}{(k-j)!j!}\\
	& = (-1)^{|\lambda|-\lambda_1}  \sum_{j = 0}^{\lambda_1} \frac{H^1_*(\lambda,j)}{\alpha^jj!} \sum_{k=j}^{\lambda_1} \frac{(-1)^{k-j}}{\alpha^{k-j}(k-j)!}\\
	& = (-1)^{|\lambda|-\lambda_1}  \frac{1}{\alpha^{\lambda_1} \lambda_1!} \sum_{j = 0}^{\lambda_1} d^{(\alpha)}_{\lambda_1,j} H^1_*(\lambda,j),
\end{align*}
which proves the result for $n = \lambda_1$. Since $\mathbf{H}^1_*(\lambda,x)$ has degree $\lambda_1$, the $n$th order forward difference $\Delta^n$ of $\mathbf{H}^1_*(\lambda,x)$ at the origin vanishes for all $n > \lambda_1$. Therefore, we have
\[
	\sum_{k=0}^{\lambda_1} \frac{1}{\alpha^k} \sum_{j =0}^k (-1)^{k-j} \frac{ H^1_*(\lambda,j) }{{ (k-j)! j!}} = \sum_{k=0}^{n} \frac{1}{\alpha^k} \sum_{j =0}^k (-1)^{k-j} \frac{ H^1_*(\lambda,j) }{{ (k-j)! j!}}
\]
for all $n \geq \lambda_1$, thus 
$$\eta_\alpha^\lambda = (-1)^{|\lambda|-\lambda_1}  \frac{1}{\alpha^{n} n!} \sum_{j = 0}^{n} d^{(\alpha)}_{n,j} H^1_*(\lambda,j),$$
as desired.
\end{proof}
\noindent Theorem~\ref{thm:dnk} allows us to connect the Jack derangement sums to the Poisson distribution. For all $\alpha \in \mathbb{R}$, a simple induction shows that $\sum_{j=0}^n d^{(\alpha)}_{n,j}/\alpha^nn! = 1$, and moreover, that
\[
	\lim_{n \rightarrow \infty} \frac{d^{(\alpha)}_{n,k}}{\alpha^nn!} = \frac{e^{-1/\alpha}}{\alpha^k k!}.
\]
For $\alpha > 0$, the limiting distribution is the Poisson distribution with expected value $1/\alpha$. After taking limits, for all $\alpha \in \mathbb{R}$, we have  
\begin{align}\label{eq:poisson}
	\eta_\alpha^\lambda = (-1)^{|\lambda|-\lambda_1} e^{-1/\alpha} \sum_{x = 0}^\infty  \frac{H^1_*(\lambda,x)}{\alpha^x x!}.
\end{align}
For $\alpha > 0$, we may interpret the Jack derangement sum as some type of ``generalized factorial moment" of the Poisson distribution (up to sign), i.e., 
\begin{align*}
	\eta_\alpha^\lambda = (-1)^{|\lambda|-\lambda_1} \mathbb{E}[\mathbf{H}^1_*(\lambda,x)].
\end{align*}
A combinatorial interpretation of these moments will follow as a corollary of Theorem~\ref{thm:main}. It is well-known that the factorial moments of the Poisson distribution have a remarkably simple form. For all $\alpha \in \mathbb{R}$, we have
\begin{align}\label{eq:fact}
\lim_{x \rightarrow \infty} \frac{x^{\underline{k}_\alpha}}{\alpha^x x!} = e^{1/\alpha}
\end{align}
where $x^{\underline{k}_\alpha} := \alpha^k x^{\underline{k}}$. In light of Equation~(\ref{eq:poisson}), the foregoing suggests that we should express the polynomial $\mathbf{H}^1_*(\lambda,x)$ in the \emph{$\alpha$-falling factorial basis} $\{x^{\underline{k}_\alpha}\}$, which we determine below for $\lambda$ such that $\lambda_1 = 1,2,3$. \\

If $\lambda_1 = 1$, then we have $\mathbf{H}^1_*(\lambda,x) = -x^{\underline{1}_\alpha} + \lambda_1'x^{\underline{0}_\alpha}$. If $\lambda_1 = 2$, then we have 
\[
	\mathbf{H}^1_*(\lambda,x) = x^{\underline{2}_\alpha} - (\lambda_2' + \lambda_1')x^{\underline{1}_\alpha}  + \lambda_2'(\alpha + \lambda_1')x^{\underline{0}_\alpha}.
\]
If $\lambda_1 = 3$, then we may write $\mathbf{H}^1_*(\lambda,x)$ as  
\[
-x^{\underline{3}_\alpha} + (\lambda_3' + \lambda_2' + \lambda_1') x^{\underline{2}_\alpha} - ((\alpha + \lambda_1')\lambda_3' + (\alpha + \lambda_1')\lambda_2' + (\alpha + \lambda_2')\lambda_3') x^{\underline{1}_\alpha}  + \lambda_3'(\alpha + \lambda_2')(2\alpha + \lambda_1').
\]
Indeed, the following proposition shows that each coefficient of $\mathbf{H}^1_*(\lambda,x)$ expressed in the $\alpha$-falling factorial basis is a polynomial $c^\lambda_k(\alpha)$ that admits a combinatorial interpretation.

\begin{proposition} \label{prop:count}
Let $\hat{\lambda}$ be the partition obtained by removing the first column of $\lambda$, and let $\emph{\#cyc}(\sigma)$ denote the number of cycles of a permutation $\sigma$. For all $\alpha \in \mathbb{R}$, we have
$$\mathbf{H}^1_*(\lambda,x) = \sum_{k = 0}^{\lambda_1} c_k^\lambda(\alpha) x^{\underline{k}_\alpha}$$ 
where $c_k^\lambda(\alpha) = (\alpha(\lambda_1-1-k) + \lambda_{1}') c_k^{\hat{\lambda}} (\alpha) - c_{k-1}^{\hat{\lambda}} (\alpha)$, $c_{k}^{\lambda}(\alpha) := 0$ if $k>\lambda_1$, $c_{-1}^{\lambda}(\alpha) := 0$. Moreover, we have 
\[
	(-1)^k[\alpha^{\lambda_1-k-j}]c_k^\lambda(\alpha) = \!\! \sum_{\substack{I \subseteq [\lambda_1] \\ |I| = k }} |\left \{ (c,\sigma) \in \mathcal{S}_\lambda : \emph{\#cyc}(\sigma) = k+j ~\emph{ and } ~ c_i = 1, \sigma(i) = i~\forall i \in I \right \}|.
\] 
\end{proposition}
\begin{proof}
By induction, we have 
\begin{align*}
	\mathbf{H}^1_*(\lambda,x) &=  (\alpha ((\lambda_1-1)-x)+\lambda_{1}') \sum_{k = 0}^{\lambda_1-1} c^{\hat{\lambda}}_k(\alpha) x^{\underline{k}_\alpha}\\
	&=  \alpha(\lambda_1 - 1 - x) \sum_{k = 0}^{\lambda_1-1} c^{\hat{\lambda}}_k(\alpha) x^{\underline{k}_\alpha} + \lambda_{1}' \sum_{k = 0}^{\lambda_1-1} c^{\hat{\lambda}}_k(\alpha) x^{\underline{k}_\alpha}\\
	&=  \sum_{k = 0}^{\lambda_1-1}(\alpha(\lambda_1 - 1-k) - \alpha(x-k)) c^{\hat{\lambda}}_k(\alpha) x^{\underline{k}_\alpha} + \lambda_1' \sum_{k = 0}^{\lambda_1-1} c^{\hat{\lambda}}_k(\alpha) x^{\underline{k}_\alpha}\\
	&=  \sum_{k = 0}^{\lambda_1-1} (\alpha(\lambda_1 - 1-k)c^{\hat{\lambda}}_k(\alpha) x^{\underline{k}_\alpha} - c^{\hat{\lambda}}_k(\alpha) x^{\underline{k+1}_\alpha}) + \lambda_1' \sum_{k = 0}^{\lambda_1-1} c^{\hat{\lambda}}_k(\alpha) x^{\underline{k}_\alpha}\\
	&=  \sum_{k = 0}^{\lambda_1-1} (\alpha(\lambda_1 - 1-k) + \lambda_1')c^{\hat{\lambda}}_k(\alpha) x^{\underline{k}_\alpha} -c^{\hat{\lambda}}_k(\alpha) x^{\underline{k+1}_\alpha})\\
	&= \sum_{k = 0}^{n-1} [(\alpha(\lambda_1 - 1-k) + \lambda_1')c^{\hat{\lambda}}_k(\alpha) - c^{\hat{\lambda}}_{k-1}(\alpha)] x^{\underline{k}_\alpha} \\
	&= \sum_{k = 0}^{\lambda_1} c_k^\lambda(\alpha) x^{\underline{k}_\alpha},
\end{align*}
which proves the first statement. To prove the second statement, note that the recurrence relation shows that $\text{sgn}~c_k^\lambda(\alpha) = (-1)^k$. The parameter $\alpha$ records the $\lambda_1 - 1 - k$ ways to join a cycle of $\hat{\lambda}$-colored permutation that is not one of $k$ singleton cycles $I \subseteq [\lambda_1] \setminus \{1\}$. There are $\lambda_1'$ ways of not joining a $\hat{\lambda}$-colored permutation. Of the latter, the choice $(1,1) \in \lambda$ results in a fixed point $1 \in I$, leaving are $k-1$ choices for the remaining elements of $I \subseteq [\lambda_1] \setminus \{1\}$. Therefore, we add $|c_{k-1}^{\hat{\lambda}}(\alpha)|$, which completes the proof. 
\end{proof}

\begin{proof}[Proof of Theorem~\ref{thm:main}]
By Equation~(\ref{eq:poisson}), it suffices to show that 
\[
	 e^{-1/\alpha} \sum_{x = 0}^\infty \frac{ H^1_*(\lambda,x)}{\alpha^x x!} = \sum_{j = 0}^{\lambda_1} d^\lambda_j \alpha^{\lambda_1 - j} = D_{\alpha}^\lambda.
\]
Recall that $c^\lambda_k(\alpha)$ = $[x^{\underline{k}_\alpha}] H^1_*(\lambda,x)$ is the $x^{\underline{k}_\alpha}$-coefficient of  $H^1_*(\lambda,x)$ expressed in the $\alpha$-falling factorial basis. By Proposition~\ref{prop:count} and Equation~(\ref{eq:fact}), we have
\begin{align*}
 e^{-1/\alpha}\sum_{x=0}^\infty \frac{H^1_*(\lambda,x)}{\alpha^x x!} &=  e^{-1/\alpha} \sum_{x = 0}^{\infty} \sum_{k=0}^{\lambda_1} \frac{c^\lambda_k(\alpha)  x^{\underline{k}_\alpha}}{\alpha^x x!} \\
 &=  e^{-1/\alpha}  \sum_{k=0}^{\lambda_1} \sum_{x = 0}^{\infty} \frac{c^\lambda_k(\alpha) x^{\underline{k}_\alpha}}{\alpha^x x!}. \\
 &= \sum_{k=0}^{\lambda_1}c^\lambda_k(\alpha).
 \intertext{By the Principle of Inclusion-Exclusion, we have}
&= \sum_{j = 0}^{\lambda_1} d^\lambda_j \alpha^{\lambda_1 - j} = D^\lambda_\alpha,
\end{align*}
which completes the proof of our first main result.
\end{proof}

\section{Proofs of Corollaries~\ref{thm:ast}, \ref{thm:kwt}, and \ref{thm:minmax}}\label{sec:cor}

With Theorem~\ref{thm:main} in hand, we now give short proofs of the corollaries stated in Section~\ref{sec:intro}.

\begin{proof}[Proof of Corollary~\ref{thm:ast}]
Clearly $D^\lambda_\alpha \geq 0$ for all $\alpha \geq 0$, so the proof follows from Theorem~\ref{thm:main}.
\end{proof}
\begin{proof}[Proof of Corollary~\ref{thm:kwt}]
By Proposition~\ref{prop:dom} and induction, it suffices to prove the result for $\mu,\lambda$ such that $\mu \nearrow \lambda$ (see Section~\ref{sec:prelim} for a review of the dominance ordering $\trianglelefteq$). Let $i$ be the column of the outer corner and let $j$ be the column of the inner corner.  Note that $i < j$.

By Theorem~\ref{thm:main}, for all $\nu$, we have $|\eta^\nu_\alpha| = \sum_{k=1}^{\nu_1} d^\nu_k \alpha^{\nu_1 - k}$, so it suffices to show that $d^\mu_k \leq d^\lambda_k$ for all $k$, i.e., that the number of colored derangements $(c,\sigma)$ with precisely $k$ disjoint cycles does not decrease when the symbol $i$ loses the color $b := \mu'_i$ and the symbol $j$ gains the color $a := \lambda'_j = \mu'_j + 1$. To show this, we give an injective map $\phi_k : \mathcal{D}^\mu_k \rightarrow \mathcal{D}^\lambda_k$ for all $k$ as follows. 

First, since $i < j$, we have $b > a$. If $c_i \neq b$, then $\phi_k(c,\sigma) = (c,\sigma) \in \mathcal{D}^\lambda_k$. If $c_i = b$, then $\phi_k(c, \sigma) = (c', (i~j) \sigma (i~j))$ where the coloring $c'$ is defined below (note that $\phi_k$ is indeed well-defined since we have $\mu_1 = \lambda_1$, i.e., $\mu$-colored and $\lambda$-colored permutations are defined on the same symbol set $[\lambda_1] = [\mu_1]$). 

Since $c_i = b$, the symbols $i$ and $j$ do not belong to the same cycle of $\sigma$, thus $c_i \neq c_j$. Also, recall that $(i~j) \sigma (i~j)$ relabels the symbols of $\sigma$ so that $i$ becomes $j$ and vice versa. Let $I$ be the set of symbols of the cycle of $\sigma$ that contains $i$. Define $c_{i'}' := a$ for all $i' \in (I \cup \{j\}) \setminus \{i\}$ so that all the symbols of $j$'s cycle in $(i~j) \sigma (i~j)$ have the same color. Define $c_i' := c_j$ so that all symbols in $i$'s cycle of $(i~j) \sigma (i~j)$ have the same color. Finally, let $c_l' := c_l$ for all remaining symbols $l \notin I \cup \{j\}$. Clearly $(i~j) \sigma (i~j)$ has the same cycle type as $\sigma$, and so it follows that $\phi_k(c, \sigma) \in \mathcal{D}^\lambda_k$. It is also clear that every $(c',\sigma')$ in the image of $\phi_k$ has a unique preimage; therefore, $\phi_k$ is injective for all $k$, as desired.
\end{proof}
Before we prove Corollary~\ref{thm:minmax}, which characterizes the extrema of the Jack derangements for $\alpha \geq 1$, we require a proposition that is essentially the Jack generalization of the well-known fact that the probability of drawing a derangement uniformly at random from $S_n$ is greater than $1/3$ for $n \geq 4$.
\begin{proposition}\label{prop:lb}
For all $\alpha \geq 1$ and $n \geq 4$, we have $D^{(n)}_\alpha > H_*^{(n)}/3$.
\end{proposition}
\begin{proof} By our main result, we have
\begin{align*}
	D^{(n)}_\alpha &= \sum_{k=0}^n (-1)^k \frac{n(n-1)\cdots(n-k+1)}{k!} H_*^{(n-k)}\\ 
	&= H_*^{(n)} \sum_{k=0}^n \frac{(-1)^k}{k!} \frac{n(n-1)\cdots(n-k+1)}{(\alpha(n-1) + 1)(\alpha(n-2) + 1) \cdots (\alpha(n-k)+1) }\\
	\intertext{For $\alpha \geq 1$ and $k > 0$, we have
	\[ 
		\frac{n(n-1)\cdots(n-k+1)}{k!(\alpha(n-1) + 1) \cdots (\alpha(n-k)+1) } - \frac{n(n-1)\cdots(n-(k+1)+1)}{(k+1)!(\alpha(n-1) + 1)\cdots (\alpha(n-(k+1))+1) }  > 0. 
	\]
Iteratively applying this fact to the $k \geq 4$ terms of the summation gives us }
		&> H_*^{(n)} [ 1 - \frac{n}{\alpha(n-1)+1} + \frac{n(n-1)}{ 2(\alpha( n-1) +1)(\alpha(n-2) +1)} \\
		&\quad\quad\quad - \frac{n(n-1)(n-2)}{6(\alpha(n-1) + 1)(\alpha (n- 2) +1)(\alpha(n - 3) + 1)}] \\
	&\geq H_*^{(n)}/3,
\end{align*}
where the last inequality follows from the fact that $\alpha \geq 1$.
\end{proof}
\begin{proof}[Proof of Corollary~\ref{thm:minmax}]
Let $\mu := (n-1,1)$. By Theorem~\ref{thm:main} and Proposition~\ref{prop:lb}, we have 
\[
	\eta^{(n)}_\alpha = D^{(n)}_\alpha,  \quad \text{ and } \quad  |\eta^{\mu}_\alpha| = D^{(n)}_\alpha /(\alpha(n-1)) > H^{(n)}_*/3(\alpha(n-1)). 
\]
Since $\alpha \geq 1$, we have $|\eta^\lambda_\alpha| \leq H^1_*(\lambda)$ by Proposition~\ref{prop:upperbound}; therefore, it suffices to show that 
\[
	H^1_*(\mu)/3 \geq H^1_*(\lambda) \quad \text{ for all $\lambda \neq (n),\mu$.}
\]
Recall that $h_\lambda(i,j) = a_\lambda(i,j) + l_\lambda(i,j) + 1$ is the hook length of $(i,j) \in \lambda$.
Define $$A := \{ h^\lambda_*(1,j) \}_{j = 1}^{\lambda_1} \quad \text{ and } \quad B := \{ h^{\mu}_*(1,j) \}_{j = 1}^{n-1}.$$ 
Note that $|A| < |B|$. Now define the injective map $\phi$ on lower hook lengths of the first row
$$\phi : A \rightarrow B \quad \text{ such that } \quad h^\lambda_*(1,j) \mapsto h^{\mu}_*(1,j')$$ 
where $j'$ is the greatest column index of $\mu$ for which $h_\lambda(1,j) \leq h_{\mu}(1,j')$. Due to the fact that $a_\lambda(1,j) \leq a_{\mu}(1,j')$, we have $h^\lambda_*(1,j) \leq h^{\mu}_*(1,j')$. 
Let $\text{im}\phi \subseteq B$ be the image of $\phi$. By the definition of $\phi$, we have $\prod_{a \in A} a \leq \prod_{b \in \text{im} \phi} b$. Since $n \geq 6$, we have $3 \leq \prod_{b \notin \text{im} \phi} b$, thus
$$ H^1_*(\lambda) = \prod_{a \in A} a \leq 3 \prod_{b \in \text{im}\phi} b \leq H^1_*(\mu)/3,$$ as required.
\end{proof}

\noindent It may be interesting to explore these corollaries for other ranges of $\alpha \in \mathbb{R}$. Computational experiments show that the Jack derangements behave quite differently when $\alpha < 0$, but perhaps there is still an elegant characterization of the sign, relative magnitude, and extrema in this range. Using Corollary~\ref{thm:kwt}, one could also try to extend Corollary~\ref{thm:minmax} to a total ordering of all the Jack derangement sums.

\section{Eigenvalues of the Permutation Derangement Graph}\label{sec:eigs}

All of the recursive expressions mentioned in Section~\ref{sec:intro} for the eigenvalues of the permutation derangement graph  embark from~\cite[Ex. 7.63a]{StanleyV201}, where Stanley considers the sum 
$$d_\lambda := \sum_{\pi \in D_n} \chi^\lambda(\pi)$$
and shows it can be expressed in terms of the complete homogeneous symmetric functions:
\[
	\sum_{\lambda \vdash n} d_\lambda s_\lambda = \sum^n_{k=0} (-1)^{n-k} n^{\underline{k}}h_1^{n-k} h_{n-k}.
\]
For hook shapes, both Stanley~\cite[Ex. 7.63b]{StanleyV201} and Okazaki~\cite[Corollary 1.3]{OkazakiPhD} prove that
\[
	d_{(j,1^{n-j})} = (-1)^{n-j} \binom{n}{j}|D_j| + (-1)^{n-1} \binom{n-1}{j} = (-1)^{n-j} \binom{n-1}{j} ( (n-j) |D_{j-1}| + |D_{j}| ).
\]
Recalling from Section~\ref{sec:intro} that $\eta_1^\lambda = d_\lambda / f^\lambda$ where $f^\lambda := \chi^\lambda(1)$ is the number of standard Young tableaux of shape $\lambda$, the following extends Stanley and Okazaki's results to all $\lambda$. 
\begin{corollary}$d_\lambda  =  (-1)^{|\lambda|-\lambda_1} f^\lambda D^\lambda$.
\end{corollary}
\noindent This suggests a natural combinatorial interpretation of $|d_\lambda|$ in terms of standard Young tableaux $t$ of shape $\lambda$ and colored derangements $(c,\sigma) \in \mathcal{D}^\lambda$. Indeed, the set $\mathcal{D}^\lambda$ is in bijection with permutations $\sigma'$ defined on $\lambda_1$ cells of a fixed Young diagram $t$ of shape $\lambda$ that satisfy the following criteria. 
\begin{enumerate}
\item If $\sigma'(i) = j$, then the cells containing $i$ and $j$ belong to the same row of $t$.
\item No two cells involved in the permutation $\sigma'$ lie in the same column of $t$.
\item If $\sigma'(i) = i$, then the cell containing $i$ does not belong to the first row of $t$.
\end{enumerate}

We obtain the desired count by letting $t$ vary over all standard Young tableaux of shape $\lambda$. 
For $\lambda = 1^n$ this gives a notably different proof of the well-known identity
\[
	d_{1^n} = \sum_{\pi \in D_n} \text{sgn}(\pi) = \sum_{\pi \in D_n}(-1)^{\text{inv}(\pi)} = (-1)^{n-1} (n-1),
\]
i.e., that the number of odd derangements versus even derangements differ by $\pm (n-1)$. 

Recall that Theorem~\ref{thm:main} gives an expression for the Jack derangement numbers as a polynomial in $\alpha$ with non-negative coefficients
\[
	D^\lambda_\alpha = d^\lambda_1 \alpha^{\lambda_1-1} + d^\lambda_2 \alpha^{\lambda_1-2} + \cdots + d^\lambda_{\lambda_1}
\]
where $d^\lambda_k$ is the number of colored permutations of $\mathcal{D}^\lambda$ that have precisely $k$ disjoint cycles. One issue with this formula is that the $d^\lambda_k$'s are hard to compute for general shapes $\lambda$, as they are at least as difficult as the associated Stirling numbers of the first kind. Theorem~\ref{thm:dnk} offers a more concrete but less combinatorial form, which for arbitrary $\alpha$ seems to be as good as it gets; however, for $\alpha = 1,2$, we show that Theorem~\ref{thm:dnk} can be massaged into an explicit combinatorial closed form in terms of what we call \emph{extended hook products}. In addition, we recover Renteln's determinantal formula for $\eta^\lambda_1$~\cite[Theorem 4.2]{Renteln07} for $\alpha = 1$. Before we begin, we require a few simple but unconventional tableau-theoretic definitions. 

Let $\lambda^c$ be the \emph{complement} of $\lambda$, defined such that
$$\lambda^c := (\lambda_1 - \lambda_1, \lambda_1 - \lambda_2, \cdots, \lambda_1 - \lambda_{\ell(\lambda)}).$$
In other words, the complement of $\lambda$ is the subset of cells of the shape $(\lambda_1)^{\ell(\lambda)}$ that do not lie in $\lambda$. For $\lambda = (10,6,3,1)$, the complement $\lambda^c = (0,4,7,9)$ is the set of dots below:
\[
	\ytableausetup{smalltableaux}
	\begin{ytableau}
	~ & ~ & ~ & ~ & ~ & ~ & ~ & ~ & ~ &~ & \none[0] \\
	~ & ~  &~ &~ & ~ & ~ & \none[\circ] & \none[\circ]& \none[\circ]& \none[\circ] & \none[4]\\ 
	~ & ~ &~& \none[\circ]& \none[\circ]& \none[\circ]& \none[\circ]& \none[\circ]& \none[\circ]& \none[\circ]& \none[7] \\
	~ & \none[\circ]& \none[\circ]& \none[\circ]& \none[\circ]& \none[\circ]& \none[\circ]& \none[\circ]& \none[\circ]& \none[\circ] & \none[~9.]
	\end{ytableau}
	\ytableausetup{nosmalltableaux}
\]
Let $\text{rev}(\lambda^c)$ be the partition obtained by reversing the order of the rows of $\lambda^c$. We also let $\text{rev} : \lambda^c \rightarrow \text{rev}(\lambda^c)$ denote the natural bijection defined on their cells, e.g., 
\[
	\ytableausetup{smalltableaux}
	\text{rev}\left(
	\begin{ytableau}
	\none[~] & \none[~]  &\none[~] & \none[~] & \none[~] & \none[~] & \none[u] & \none[t]& \none[s]& \none[r]\\ 
	\none[~] & \none[~] & \none[~] & \none[q]& \none[p]& \none[o]& \none[n]& \none[m]& \none[l]& \none[k] \\
	\none[j] & \none[i]& \none[h]& \none[g]& \none[f]& \none[e]& \none[d]& \none[c]& \none[b]& \none[a]
	\end{ytableau}
	\right)
	\quad = \quad 
	\begin{ytableau}
	\none[a] & \none[b] & \none[c] & \none[d] & \none[e] & \none[f] & \none[g] & \none[h] & \none[i] & \none[j] \\
	\none[k] & \none[l]  &\none[m] & \none[n] & \none[o] & \none[p] & \none[q] \\ 
	\none[r] & \none[s] & \none[t] & \none[~u~] \\
	\none[~]
	\end{ytableau}~~.
	\ytableausetup{nosmalltableaux} 
\]
For any cell $\Box \in \lambda^c$, we define its \emph{upper hook length} to be $h_{\lambda^c}^*(\Box) =  h_{\text{rev}(\lambda^c)}^*(\text{rev}(\Box))$, and similarly for lower hook lengths. For example, we have the following upper hook lengths for $\alpha = 1$ and $\mu = (10,6,3,1)$:

\ytableausetup{mathmode, boxsize=1.5em}
\[
	\begin{ytableau}
	13 & 11 & 10 & 8 & 7 & 6 & 4 & 3 & 2 & 1 \\
	8 & 6  & 5 & 3 & 2 & 1 & \none[1] & \none[2]& \none[3]& \none[4] \\ 
	4 & 2 & 1 & \none[1]& \none[2]& \none[3]& \none[5]& \none[6]& \none[7]& \none[8] \\
	1 & \none[1]& \none[2]& \none[4]& \none[5]& \none[6]& \none[8]& \none[9]& \none[~10]&\none[~~11]
	\end{ytableau}~~.
\]
We define the \emph{extended ith principal upper hook product} as follows:
\[
	{H}_i^{+}(\lambda) :=  H_i^*(\lambda) H_{i}^*(\lambda^c).
\]
Continuing the example above, we see that ${H}_3^+(\mu) = 4 \cdot 2 \cdot 1 \cdot 8!/4 = 80640$. Note that ${H}_1^*(\lambda) = H_1^+(\lambda)$ for all $\lambda$ since $(\lambda^c)_1 = 0$.

Let $p(\lambda) = p_0,p_1, \ldots, p_{\lambda_1}$ be the sequence of the first $\lambda_1+1$ edges along the NE--SW lattice path induced by $\lambda$, e.g., 
\[
	\begin{ytableau}
	~ & ~ & ~ & ~ & ~ & ~ & ~ & ~ & ~ & ~~\downarrow \\
	~ & ~  &~ &~ & ~ & ~~\downarrow & \none[\raisebox{5pt}{$\leftarrow$}] & \none[\raisebox{5pt}{$\leftarrow$}]  &  \none[\raisebox{5pt}{$\leftarrow$}] &  \none[\raisebox{5pt}{$\leftarrow$}] \\ 
	~ & ~ & ~~\downarrow &  \none[\raisebox{5pt}{$\leftarrow$}]  &  \none[\raisebox{5pt}{$\leftarrow$}] &  \none[\raisebox{5pt}{$\leftarrow$}]  \\
	~ & \none[~]&  \none[\raisebox{5pt}{$\leftarrow$} ]
	\end{ytableau}~~ .
\]
Let $\nu(\lambda) := \nu_1,\nu_2 ,\ldots,\nu_l$ be the indices of the subsequence of vertical edges of $p$. It is not difficult to see that
$$
\nu(\lambda) = ( \lambda_1 - \lambda_i + i - 1 : i - 1  \leq \lambda_i).
$$
Continuing the example, we have $\nu(\mu) = 0,5,9$. Note that $\nu(\lambda)_1 = 0$ for all $\lambda$. 

Recall from Section~\ref{sec:main} that $d_{n,k}$ is the \emph{$k$th rencontres number}, i.e., the number of permutations of $S_n$ with precisely $k$ fixed points. Let $p_{n,k} = d_{n,k}/n!$ be the probability of drawing a permutation (uniformly at random) from $S_n$ with precisely $k$ fixed points. The \emph{Frobenius coordinates} of $\lambda$ are given by $\lambda = (a_1,\ldots,a_d~|~b_1,\ldots,b_d)$ where $a_i := \lambda_i - i$ is the number of boxes to the right of the diagonal in row $i$, and $b_i := \lambda_i' -i$ is the number of boxes below the diagonal in column $i$. By default, we define $a_{d+1} := -1$. We are now ready to give a nice closed form for the eigenvalues of the permutation derangement graph $\Gamma_{n,1}$.

\begin{figure}
\tiny
	\[
	\ytableausetup{mathmode, boxsize=3em}
	\begin{ytableau}
	4\!-\!\alpha & 3\!-\!2\alpha & 3\!-\!3\alpha & 2\!-\!4\alpha & 2\!-\!5\alpha & 2\!-\!6\alpha & 1\!-\!7\alpha & 1\!-\!8\alpha & 1\!-\!9\alpha & 1\!-\!10\alpha & \none[~j=10]\\
	*(cyan)4 & *(cyan) 3\!-\!\alpha & *(cyan) 3\!-\!2\alpha & 2\!-\!3\alpha & 2\!-\!4\alpha & 2\!-\!5\alpha & 1\!-\!6\alpha & 1\!-\!7\alpha & 1\!-\!8\alpha & 1\!-\!9\alpha & \none[j=9]\\
	4\!+\!\alpha & 3 & 3\!-\!\alpha & 2\!-\!2\alpha & 2\!-\!3\alpha & 2\!-\!4\alpha & 1\!-\!5\alpha & 1\!-\!6\alpha & 1\!-\!7\alpha & 1\!-\!8\alpha & \none[j=8]\\
	4\!+\!2\alpha & 3\!+\!\alpha & 3 & 2\!-\!\alpha & 2\!-\!2\alpha & 2\!-\!3\alpha & 1\!-\!4\alpha & 1\!-\!5\alpha & 1\!-\!6\alpha & 1\!-\!7\alpha & \none[j=7]\\
	4\!+\!3\alpha &  3\!+\!2\alpha & 3\!+\!\alpha & 2 & 2\!-\!\alpha & 2\!-\!2\alpha & 1\!-\!3\alpha & 1\!-\!4\alpha & 1\!-\!5\alpha & 1\!-\!6\alpha & \none[j=6]\\
	*(green) 4\!+\!4\alpha & *(green) 3\!+\!3\alpha &  *(green) 3\!+\!2\alpha & *(green) 2\!+\!\alpha & *(green) 2 & *(green) 2\!-\!\alpha & 1\!-\!2\alpha & 1\!-\!3\alpha & 1\!-\!4\alpha & 1\!-\!5\alpha & \none[j=5]\\
	4\!+\!5\alpha & 3\!+\!4\alpha & 3\!+\!3\alpha &  2\!+\!2\alpha & 2\!+\!\alpha & 2 & 1\!-\!\alpha & 1\!-\!2\alpha & 1\!-\!3\alpha & 1\!-\!4\alpha & \none[j=4]\\
	4\!+\!6\alpha & 3\!+\!5\alpha & 3\!+\!4\alpha & 2\!+\!3\alpha &  2\!+\!2\alpha & 2\!+\!\alpha & 1 & 1\!-\!\alpha & 1\!-\!2\alpha & 1\!-\!3\alpha & \none[j=3]\\
	4\!+\!7\alpha & 3\!+\!6\alpha & 3\!+\!5\alpha & 2\!+\!4\alpha & 2\!+\!3\alpha &  2\!+\!2\alpha & 1\!+\!\alpha & 1 & 1\!-\!\alpha & 1\!-\!2\alpha & \none[j=2]\\
	4\!+\!8\alpha & 3\!+\!7\alpha & 3\!+\!6\alpha &  2\!+\!5\alpha & 2\!+\!4\alpha & 2\!+\!3\alpha &  1\!+\!2\alpha & 1\!+\!\alpha & 1 & 1\!-\!\alpha & \none[j=1]\\
	*(yellow) 4\!+\!9\alpha & *(yellow) 3\!+\!8\alpha & *(yellow) 3\!+\!7\alpha & *(yellow) 2\!+\!6\alpha & *(yellow) 2\!+\!5\alpha & *(yellow) 2\!+\!4\alpha & *(yellow) 1\!+\!3\alpha & *(yellow) 1\!+\!2\alpha & *(yellow) 1\!+\!\alpha & *(yellow) 1 & \none[j=0]\\
	*(green) 3\!+\!5\alpha & *(green) 2\!+\!4\alpha  & *(green) 2\!+\!3\alpha & *(green) 1\!+\!2\alpha & *(green) 1\!+\!\alpha & *(green) 1 \\ 
	*(cyan) 2\!+\!2\alpha & *(cyan)1\!+\!\alpha &*(cyan) 1 \\
	*(pink) 1  \\
	\end{ytableau}
	\]
		\normalsize
	\caption{The shifted principal lower hook products for $\lambda = (10,6,3,1)$. Row $j$ shows the product $H^1_*(\lambda,j)$. 
For $\alpha = 1$, the product of the hook lengths along any uncolored row is 0, the product of the hook lengths along row 2 of $\lambda$ equals the product of the green cells in the $j=5$ row, and the product of the hook lengths along row 3 of $\lambda$ equals the product of the blue cells in the $j=9$ row. 
Theorem~\ref{thm:eigen1} shows $\eta_1^\lambda = p_{10,0} [13!/(12 \cdot 9 \cdot 5)] + p_{10,5} [8!4!/(7 \cdot 4)] - p_{10,9}H_3^+(\lambda) = 4242315$ (note $p_{10,9}=0$). For $\alpha = 2$, we have $H^1_*(\lambda,j) = 0$ for $j = 5,6,7$.
}\label{fig:board}
\end{figure}

\begin{theorem}[Eigenvalues of $\Gamma_{n,1}$]\label{thm:eigen1} For all $\lambda = (\lambda_1,\ldots, \lambda_\ell) = (a_1,\ldots,a_d~|~b_1,\ldots,b_d) \vdash n$, we have
	\begin{align*}
		\eta^\lambda_1 
		&= (-1)^{n} \sum_{i \leq \lambda_i + 1} (-1)^{\lambda_i}  p_{\lambda_1, a_1 - a_i} ~H_{i}^+(\lambda)
	\end{align*}
\end{theorem}
\begin{proof}
The product $\prod_{k=1}^{\lambda_1}(h_k - i)$ vanishes if $h_k = i$ for some $k$, which happens if and only if $i \notin \nu(\lambda)$ (see Figure~\ref{fig:board} for an illustration). Otherwise, we have $H_i^+(\lambda) = |\prod_{k=1}^{\lambda_1}(h_k - i)|$ and $\text{sgn}~\prod_{k=1}^{\lambda_1}(h_k - i) = (-1)^{\lambda_1 - \lambda_i}$. The proof now follows from Theorem~\ref{thm:dnk}.
\end{proof}

\begin{corollary}\emph{\cite{DengZ11}} For all two-row shapes $\lambda = (n-k,k)$, we have
\[
		\eta^\lambda_1 
		=  \frac{  (-1)^{k} d_{n-k+1,1} + (-1)^{n-k} d_{k,1} }{(n-2k+1)}.
\]
\end{corollary}
\begin{proof} 
\begin{align*}
	\eta^{(n-k,k)}_1 &= \frac{1}{(n-k)!} \left(  (-1)^{k} d_{n-k,0} H_1(\lambda) + (-1)^{n-k} d_{n-k,n-2k+1} H_2^+ (\lambda)\right)\\
	 &=  \frac{1}{(n-k)!} \left(  (-1)^k d_{n-k,0} \frac{(n-k+1)!}{(n-2k+1)} + (-1)^{n-k}d_{n-k,n-2k+1} k!(n-2k)!  \right)\\
	 &=   (-1)^k d_{n-k,0} \frac{(n-k+1)}{(n-2k+1)} +  (-1)^{n-k} \frac{d_{n-k,n-2k+1}}{\binom{n-k}{k}} \\
	 &=  (-1)^k \frac{d_{n-k+1,1}}{(n-2k+1)} +  (-1)^{n-k} \frac{\binom{n-k}{k-1}d_{k-1,0}}{\binom{n-k}{k}} \\
	  &= \frac{ (-1)^k d_{n-k+1,1}+  (-1)^{n-k} d_{k,1} }{(n-2k+1)}.
\end{align*}
\end{proof}
\noindent For ease of notation, let $d_{n,k}' := k!d_{n,k}$ (c.f.~\emph{the shifted derangement number}~\cite[\S 4]{Renteln07}).
\begin{corollary} For all three-row shapes $\lambda = (\lambda_1,\lambda_2,\lambda_3)$, we have
\[
		\eta^{\lambda} 
		=   \frac{  (-1)^{|\lambda|-\lambda_1} d'_{\lambda_1+2,2}}{(\lambda_1-\lambda_3+2)(\lambda_1-\lambda_2+1)} +
		 \frac{  (-1)^{|\lambda|-\lambda_2} d'_{\lambda_2+1,2}}{(\lambda_1-\lambda_2+1)(\lambda_2-\lambda_3+1)}  +
		  \frac{ (-1)^{|\lambda|-\lambda_3} d'_{\lambda_3,2}}{(\lambda_1-\lambda_3+2)(\lambda_2-\lambda_3+1)} .
\]
\end{corollary}
\begin{proof}
\begin{align*}
	\eta^{\lambda}_1 &= \frac{(-1)^{|\lambda|}}{\lambda_1!} \left(  (-1)^{\lambda_1} d_{\lambda_1,0} H_1(\lambda) + (-1)^{\lambda_2} d_{\lambda_1,\lambda_1 - \lambda_2+1} H_2^+ (\lambda) + (-1)^{\lambda_3} d_{\lambda_1,\lambda_1 - \lambda_3+2} H_3^+ (\lambda)\right)\\
	 &=   \frac{(-1)^{|\lambda|-\lambda_1} d_{\lambda_1,0}(\lambda_1 + 2)!}{\lambda_1!(\lambda_1-\lambda_2+1)(\lambda_1-\lambda_3+2)} +  \frac{(-1)^{|\lambda|- \lambda_2} d_{\lambda_1,\lambda_1 - \lambda_2+1} (\lambda_2+1)!(\lambda_1-\lambda_2)!}{\lambda_1!(\lambda_2-\lambda_3+1)}  \\
	 &\quad \quad \quad \quad +  \frac{(-1)^{|\lambda|- \lambda_3} d_{\lambda_1,\lambda_1 - \lambda_3+2} \lambda_3!(\lambda_1-\lambda_3+1)!}{\lambda_1!(\lambda_2-\lambda_3+1)}  \\
	 &=   \frac{(-1)^{|\lambda|-\lambda_1} d_{\lambda_1+2,2}}{(\lambda_1-\lambda_2+1)(\lambda_1-\lambda_3+2)} +   \frac{  (-1)^{|\lambda|-\lambda_2} d'_{\lambda_2+1,2}}{(\lambda_1-\lambda_2+1)(\lambda_2-\lambda_3+1)}  \\
	 &\quad \quad \quad \quad +   \frac{ (-1)^{|\lambda|-\lambda_3} d'_{\lambda_3,2}}{(\lambda_1-\lambda_3+2)(\lambda_2-\lambda_3+1)}.
\end{align*}
\end{proof}
\noindent Continuing in this manner, the expression above becomes exceedingly more cumbersome to explicitly write down for partitions with more parts; however, it does suggest a compact expression as a determinant. Let $\ell := \ell(\lambda)$ and define the following $\ell \times \ell$ matrices:
\[
	W(\lambda) := \begin{bmatrix} 
	(-1)^{\lambda_1 - \lambda_1 + 1 - 1} d'_{\lambda_1 + \ell - 1, \ell-1} & (\lambda_{1} - 1)^{\ell-2} & (\lambda_{1} - 1)^{\ell-3} & \cdots & 1\\
	(-1)^{\lambda_1 - \lambda_2 + 2 - 1} d'_{\lambda_2 + \ell - 2, \ell-1} & (\lambda_{2} - 2)^{\ell-2} & (\lambda_{2} - 2)^{\ell-3} & \cdots & 1\\
	(-1)^{\lambda_1 - \lambda_3 + 3 - 1} d'_{\lambda_3 + \ell - 3, \ell-1} & (\lambda_{3} - 3)^{\ell-2} & (\lambda_{3} - 3)^{\ell-3} & \cdots & 1\\
	\vdots & \vdots & \vdots & \ddots &~
	\end{bmatrix},
\]
and $V(\lambda) := ((\lambda_{i} - i)^{j-1})_{i,j=1}^{\ell}$.
Clearly $V(\lambda)$ is Vandermonde in the variables  $x_i = \lambda_i - i$, and any submatrix of $W(\lambda)$ obtained by removing the first column and then removing any row is also Vandermonde. We are now ready to show that $\eta^\lambda_1$ is a determinant. 
\begin{theorem}\emph{\cite[Theorem 4.2]{Renteln07}} For all shapes $\lambda = (\lambda_1,\lambda_2,\ldots, \lambda_\ell)$, we have
\[
		\eta^\lambda_1 
		=   \det W(\lambda)V(\lambda)^{-1}
\]
\end{theorem}
\begin{proof}
First, we claim that $H_+^k(\lambda)$ can be written as a scaled ratio of Vandermonde determinants, i.e.,
\begin{align*}
	H^k_+(\lambda) &= (\lambda_k - k + \ell)!(\lambda_1 - \lambda_k + k - 1)! \frac{\prod_{\substack{i < j \\ i,j \neq k}} (\lambda_i - \lambda_j + j-i)}{\prod_{i < j} (\lambda_i - \lambda_j + j-i)}\\ 
	&= \frac{(\lambda_k - k + \ell)!(\lambda_1 - \lambda_k + k - 1)! }{\prod_{\!\!\!\! \substack{i < j \\ i = k \text{ or } j = k}} (\lambda_i - \lambda_j + j-i)}.
\end{align*}
It is clear that the $k$th extended principal hook product of any shape cannot be larger than $(\lambda_k - k + \ell)!(\lambda_1 - \lambda_k + k - 1)!$. Here, we think of $(\lambda_k - k + \ell)!$ as representing all possible hook lengths in the cells of $\lambda_k$, i.e., $(\lambda_k - k + \ell)! \leq H_k^*(\lambda)$, and $(\lambda_1 - \lambda_k + k - 1)!$ as representing all possible hook lengths in the non-cells to the right of $\lambda_k$, i.e., $(\lambda_1 - \lambda_k + k - 1)! \leq H_k^*(\lambda^c)$. The denominator corrects for the hook lengths that do not appear in $H_k^*(\lambda)$ and $H_k^*(\lambda^c)$. Indeed, when $i < j = k$, the values $(\lambda_i - \lambda_k + k - i)$ are the only hook lengths that do not appear in $H_k^+(\lambda)$ for all $i$. Similarly, when $k = i  < j$, the values $(\lambda_k - \lambda_j + j - k)$ are the only hook lengths that do not appear in $H_k^+(\lambda^c)$ for all $j$, which proves the claim.

By Theorem~\ref{thm:dnk}, we have
\begin{align*}
	\eta^\lambda_1 &= \frac{(-1)^{\lambda_1}}{\lambda_1!} \sum_{k \leq \lambda_k + 1}(-1)^{\lambda_k}  d_{\lambda_1, \lambda_1 - \lambda_k + k - 1} ~\frac{(\lambda_k - k + \ell)!(\lambda_1 - \lambda_k + k - 1)! }{\prod_{\substack{i < j \\ i = k \text{ or } j = k}} (\lambda_i - \lambda_j + j-i)}\\
	&= \frac{(-1)^{\lambda_1}}{\lambda_1!} \sum_{k \leq \lambda_k + 1}(-1)^{\lambda_k} \binom{\lambda_1}{\lambda_1 - \lambda_k + k -1} d_{\lambda_k - k + 1, 0} ~\frac{(\lambda_k - k + \ell)!(\lambda_1 - \lambda_k + k - 1)! }{\prod_{\substack{i < j \\ i = k \text{ or } j = k}} (\lambda_i - \lambda_j + j-i)}\\
	&= (-1)^{\lambda_1} \sum_{k \leq \lambda_k + 1}(-1)^{\lambda_k} \frac{d_{\lambda_k - k + 1, 0}}{(\lambda_k - k +1)!}  ~\frac{(\lambda_k - k + \ell)!}{\prod_{\substack{i < j \\ i = k \text{ or } j = k}} (\lambda_i - \lambda_j + j-i)}\\
	&= (-1)^{\lambda_1} \sum_{k \leq \lambda_k + 1}(-1)^{\lambda_k} \binom{\lambda_k - k + \ell}{\ell-1} (\ell-1)!~\frac{d_{\lambda_k - k + 1, 0} }{\prod_{\substack{i < j \\ i = k \text{ or } j = k}} (\lambda_i - \lambda_j + j-i)}\\
	&= \sum_{k \leq \lambda_k + 1}(-1)^{k} ~\left((-1)^{\lambda_1 - \lambda_k + k} d_{\lambda_k - k + \ell, \ell-1}' \right)\left( \frac{1}{\prod_{\substack{i < j \\ i = k \text{ or } j = k}} (\lambda_i - \lambda_j + j-i)} \right).\\
	\intertext{Let $W(\lambda)^{k,1}$ be the submatrix of $W(\lambda)$ obtained by removing the first column and $k$th row. Then we have }
	&= \frac{1}{\det V(\lambda)} \sum_{k \leq \lambda_k + 1}(-1)^{k+1} ~\left((-1)^{\lambda_1 - \lambda_k + k - 1} d_{\lambda_k - k + \ell, \ell-1}' \right)~\det W(\lambda)^{k,1}.\\
	\intertext{Since $W(\lambda)_{k,1} = (-1)^{\lambda_1 - \lambda_k + k - 1} d_{\lambda_k - k + \ell, \ell-1}' $, the summation is simply the Laplace expansion of $W(\lambda)$ along the first column, that is,}
	&= \frac{\det W(\lambda)}{\det V(\lambda)} = \det W(\lambda) V(\lambda)^{-1},
\end{align*}
which completes the proof. 
\end{proof}

\noindent A standard result in the theory of symmetric functions is that $f^\lambda$ can be expressed as a determinant via the Jacobi--Trudi identity (see~\cite[Cor.~7.16.3]{StanleyV201}), which gives a determinantal expression for~\cite[Ex.~7.63a]{StanleyV201}.
\begin{corollary}\emph{\cite{Renteln07}} For all shapes  $\lambda = (\lambda_1, \lambda_2, \ldots,\lambda_\ell)$, we have 
$$d_\lambda = |\lambda| !~ \det \left( \frac{1}{\lambda_i - i + j} \right)_{i,j=1}^\ell ~ \det W(\lambda) V(\lambda)^{-1} .$$
\end{corollary}
\noindent However, the fact that $d_\lambda$ can be written as the determinant of a $\ell \times \ell$ matrix actually goes back to a result of Goulden and Jackson concerning determinantal expressions of \emph{immanants}, a $\lambda$-generalization of the permanent and determinant defined as follows:
$$\text{Imm}_\lambda(A) := \sum_{\pi \in S_n} \chi^\lambda(\pi) A_{i,\pi(i)}$$
where $A$ is any $n \times n$ matrix. Indeed, if we consider the adjacency matrix of the complete graph $K_n = J_n - I_n$ where $J_n$ is the $n \times n$ all-ones matrix, then we have
\[
	\text{Imm}_\lambda(K_n) = \sum_{\pi \in S_n} \chi^\lambda(\pi) \prod_{i=1}^n(K_n)_{i,\pi(i)} =  \sum_{\pi \in D_n} \chi^\lambda(\pi) = d_\lambda.
\]
Via the MacMahon master theorem, Goulden and Jackson~\cite[Theorem 2.1]{GouldenJ92} produce a $\ell \times \ell$ matrix $A'$ for which $\text{Imm}_\lambda(K_n) = \det A'$. The foregoing shows that this determinant has a natural combinatorial interpretation that can also be evaluated efficiently. 

We note that the original Ku--Wales theorem (see Section~\ref{sec:intro}) appears to be somewhat related to a dominance result of Pate~\cite{Pate92} on normalized immanant inequalities of positive semi-definite Hermitian matrices, a classical subject initiated by Schur. In particular, Pate shows that $ \text{Imm}_\mu(B)/f^\mu \leq \text{Imm}_\lambda (B)/f^\lambda$ for all $\mu \nearrow \lambda$ such that $\mu_{\ell(\mu)} = 1$ and $B$ is positive semi-definite. James~\cite{James92} shows for any $\mu,\lambda \vdash n$, that if $ \text{Imm}_\mu(B)/f^\mu \leq \text{Imm}_\lambda (B)/f^\lambda$, then $\mu \trianglelefteq \lambda$, thus Pate's result is a partial converse (the full converse is known to be false). We are unsure how exactly the Ku--Wales theorem fits into this literature, as $K_n$ is not positive semi-definite; nevertheless, it is curious that the absolute values of its immanants still obey an immanant dominance property with respect to the dominance ordering on partitions.

For each $\lambda \vdash n$, let $\text{Imm}_\lambda(xI - A)/f^\lambda$ be the \emph{(normalized) immanantal polynomial} of $A$.
Without much additional effort we can derive explicit formulas for the coefficients of $\text{Imm}_\lambda(xI - K_n)$, which are also of combinatorial significance. Indeed, for each $0 \leq k \leq n$, it is not difficult to show that the coefficient of $(-1)^{n-k}x^k$ is the $\lambda$-eigenvalue of the \emph{$k$-derangement graph}, i.e., the Cayley graph of $S_n$ generated by all permutations with precisely $k$ fixed points, which has been the subject of many papers in algebraic graph theory.
\begin{theorem} For all $\lambda \vdash n$, we have
\[
	\emph{Imm}_\lambda(xI - K_n)/f^\lambda = \sum_{k=0}^n  (-1)^{n-k} \left[ \sum_{\mu \nearrow^k \lambda } f^\mu\frac{s^\star_\mu(\lambda)}{(n-k)!} \eta^\mu_1 \right] x^k.
\]
\end{theorem}
\begin{proof}
 Recall that the $s_\lambda^\star$'s are the shifted Schur polynomials, i.e., the unnormalized shifted Jack polynomials at $\alpha=1$, and that $f^\lambda$ is the number of standard Young tableaux of shape $\lambda$. By the definition of the immanantal polynomial, we have
\begin{align*}
\text{Imm}_\lambda(xI - K_n)/f^\lambda &= \frac{1}{f^\lambda} \sum_{\pi \in S_n} \chi^\lambda(\pi) \prod_{i=1}^n(xI-K_n)_{i,\pi(i)}\\
\intertext{For any $k$-set $I \subseteq \binom{[n]}{k}$, let $S_n^I \subseteq S_n$ be the set of permutations such that $\sigma(i)= i$ for all $i \in I$ and $\sigma(j) \neq j$ for all $j \notin I$.}
&= \frac{1}{f^\lambda} \sum_{k=0}^n\sum_{I \subseteq \binom{[n]}{k}} \sum_{\sigma \in S_n^I }x^k (-1)^{n-k} \chi^\lambda(\sigma) \\
\intertext{For any character $\chi$ of $S_n$, let $\chi \! \! \downarrow_{S_{n-k}}$ denote the restriction to the subgroup $S_{n-k}$.}
&= \frac{1}{f^\lambda}\sum_{k=0}^n \binom{n}{k} \sum_{\pi \in D_{n-k}}x^k(-1)^{n-k} \chi^\lambda \!\! \downarrow_{S_{n-k}}(\pi) \\
\intertext{To compute this restriction we iterate the branching rule $k$ times (see \cite{Sagan01}, for example). It is well-known that the multiplicity of $\mu \vdash (n-k)$ in the restriction of $\lambda$ to $S_{n-k}$ is $f^{\lambda/\mu}$, the number of standard skew tableaux of skew shape $\lambda / \mu$, equivalently, the number of distinct ways of successively adding $k$ outer corners to obtain $\lambda$ starting from $\mu$. This gives}
&= \frac{1}{f^\lambda}\sum_{k=0}^n x^k(-1)^{n-k}\binom{n}{k} \sum_{\mu \nearrow^k \lambda } f^{\lambda/\mu} \sum_{\pi \in D_{n-k}} \chi^\mu(\pi) \\
&= \frac{1}{f^\lambda}\sum_{k=0}^n (-1)^{n-k} \binom{n}{k} \sum_{\mu \nearrow^k \lambda } f^{\lambda/\mu} f^\mu \eta^\mu_1 x^k\\
\intertext{where $\mu$ ranges over all shapes on $n-k$ cells obtained by removing $k$ outer corners successively from $\lambda$. Note that $\sum_{\mu \nearrow^k \lambda } f^{\lambda/\mu} f^\mu = f^\lambda$, thus the coefficients are convex combinations of $\mu$-eigenvalues of $\Gamma_{n-k,1}$. By~\cite[Proposition 5.2]{OkounkovO97}, we may write }
&= \sum_{k=0}^n  (-1)^{n-k}\sum_{\mu \nearrow^k \lambda } f^\mu \frac{s^\star_\mu(\lambda)}{(n-k)!}  ~\eta^\mu_1 x^k,
\end{align*}
which completes the proof.
\end{proof}
\noindent For small $k$ we obtain reasonable expressions as positive linear combinations of eigenvalues of $\Gamma_{n-k,1}$; however, these formulas quickly become unwieldy as $k$ increases. On the other hand, when $k$ is close to $n$, these coefficients can also be efficiently computed through other means, as $|\mu|$ is small. For example, the coefficient of $x^{n-2}$ is the $\lambda$-eigenvalue of the well-known \emph{transposition graph}, i.e., the Cayley graph of $S_n$ generated by all its transpositions. It would be interesting to obtain more explicit expressions for $[x^k]~\text{Imm}_\lambda(xI - K_n)/f^\lambda $. One barrier is that nice expressions for $s^\star_\mu(\lambda)/(n-k)!$ are only known in special cases, not for arbitrary $\mu \subseteq \lambda$, which itself is an open question. For more details on the $k$-derangement graphs, we refer the reader to the recent survey~\cite{LiuZ22}.

\section{Eigenvalues of the Perfect Matching Derangement Graph}\label{sec:eigs2}

We now move onto the perfect matching derangement graph, i.e., the case where $\alpha = 2$. The situation here is complicated by the fact that the upper and lower hook lengths do not coincide. We first consider the two row case, which has an interesting connection to derangements of the \emph{hyperoctahedral group} $B_n = \mathbb{Z}_2 \wr S_n \leq S_{2n}$, that is, the automorphism group of the hypercube $\{\pm 1\}^n$. Below we recall some results of Chen and Stanley~\cite{ChenS93} that will give a combinatorial interpretation of the two-row Jack derangement sums for $\alpha = 2$.

The elements $w$ of $B_n$ can be represented as \emph{signed permutations}, i.e., a permutation of $[n]$ along with a plus or minus sign attached to each symbol. To represent the signing, we adopt the shorthand $\bar{i} := i^-$ and $i := i^+\!$, e.g., $(2,4,\bar{5})(\bar{3})(1,\bar{6}) \in B_6$. Following Chen and Stanley~\cite{ChenS93}, we say that $w \in B_n$ is \emph{balanced} if each of its cycles has an even number of minus signs. For example, the element $(\bar{2},4,\bar{5})({3})(\bar{1},\bar{6})$ is balanced, whereas $(\bar{2},4,\bar{5})(\bar{3})(1,{6})$ is unbalanced. By~\cite[Cor.~2.4]{ChenS93}, the number of balanced elements of $B_n$ equals $(2n-1)!!$. We say $w \in B_n$ is \emph{totally unbalanced} if each of its cycles are unbalanced. By~\cite[Prop.~3.1]{ChenS93}, the number of totally unbalanced elements of $B_n$ also equals $(2n-1)!!$.  Chen and Stanley define an element $w \in B_n$ to be \emph{$k$-separable} if the cycles of $w$ can be partitioned into two parts $A,B$ such that every cycle of $A$ is balanced and the sum of the cycle lengths of the cycles in $B$ equals $k$. Moreover, they show that these are precisely the elements of $B_n$ that fix some $k$-dimensional subcube $\{\pm1\}^k$ of $\{\pm 1\}^n$~\cite[Prop.~2.2]{ChenS93}. Note that if $k=0$, then a 0-dimensional subcube is taken to be a vertex of the hypercube. 

Let $\mathcal{E}_n$ be the set of derangements (fixed-point-free elements) of $B_n$, i.e.,
\[
	\mathcal{E}_n = \{ w \in B_n : w(j) \neq j \text{ for all } j \in [n] \}.  
\] 
It is well-known that $|\mathcal{E}_n| = d_{n,0}^{(2)}$. Every totally unbalanced element of $B_n$ belongs to $\mathcal{E}_n$. The combinatorial proof given on~\cite[pg.~70]{ChenS93} extends to a bijection between balanced signed permutations of $B_n$ and $\mathcal{M}_{2n}$ that also maps fixed-point-free balanced signed permutations of $B_n$ to \emph{perfect matching derangements} of $\mathcal{M}_{2n}$:
$$\mathcal{D}_{2n}' := \{m \in \mathcal{M}_{2n} : m \cap m^* = \emptyset\}$$ 
where $m^* = \{\{1,\bar{1}\},\{2,\bar{2}\}, \cdots \{n,\bar{n}\}\}$.  
Finally, recall that our combinatorial proof in Section~\ref{sec:php} of Theorem~\ref{thm:main} at $\alpha = 2$ shows $|\eta^\lambda_2| = |\mathcal{D}_\lambda'|$ where $\mathcal{D}_\lambda'$ is the set of derangements of $\lambda$-colored perfect matchings $\mathcal{M}_\lambda$. 

The foregoing observations give an interesting interpretation of $\eta_2^\lambda$ for two-row shapes $(n-k,k)$ in terms of $B_{n-k}$-derangements that stabilize a fixed hypercube $\{\pm 1\}^k \subseteq \{\pm 1\}^{n-k}$. 
\begin{theorem}
For all two-row shapes $\lambda = (n-k,k)$, we have
\begin{align*}
	\eta^\lambda_2 &= (-1)^k \sum_{i=0}^{k}  \binom{k}{i}(2i-1)!! ~|\mathcal{D}_{2(n-k-i)}'|\\
	 &= (-1)^{k}~ |\{ \sigma \in \mathcal{E}_{n-k} : \sigma \text{ fixes $\{\pm 1\}^{k}$} \}|.
\end{align*}
\end{theorem}
\noindent Note that when $n$ is even and $k = n/2$, we have $\eta^{(n/2,n/2)}_2 = (-1)^{n/2} |\mathcal{E}_{n/2}|$, hence the two-row shapes interpolate between the derangements that stabilize a fixed vertex of the hypercube and the derangements that stabilize the whole hypercube. 

Recall that for $\alpha = 1$ we ignored all indices $j$ corresponding to $\leftarrow$ moves in the lattice path induced by $\lambda$, since $H_*^1(\lambda,j) = 0$ in these cases. This is no longer the case for $\alpha = 2$; however, we can still identify the non-vanishing terms via lattice paths. Here, instead of each vertical move $\downarrow$ descending by a single row, each vertical move descends by two rows, and as before we ignore horizontal moves $\leftarrow$ if they border a row of $\lambda$. 
For example, if $\lambda = (10,6,3,1)$, then we ignore the indices $j = 5,6,7$ corresponding to arrows that border the second row (see Figure~\ref{fig:board} at $\alpha = 2$):
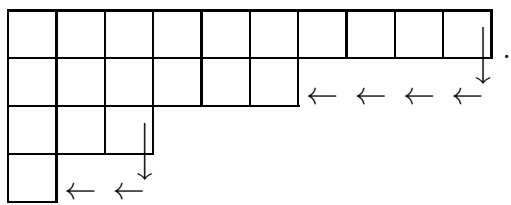
\[
\ytableausetup{mathmode, boxsize=1.5em}
\begin{ytableau}
~ & ~ & ~ & ~ & ~ & ~ & ~ & ~ & ~ & ~~~\raisebox{-8pt}{\Big\downarrow} \\
~ & ~  &~ &~ & ~ & ~ & \none[\raisebox{-5pt}{$\leftarrow$}] & \none[\raisebox{-5pt}{$\leftarrow$}]  &  \none[\raisebox{-5pt}{$\leftarrow$}] &  \none[\raisebox{-5pt}{$\leftarrow$}] \\ 
~ & ~ & ~~~\raisebox{-8pt}{\Big\downarrow}  &  \none[~]  &  \none[~] &  \none[~]  \\
~ & \none[\raisebox{-5pt}{$\leftarrow$}]&  \none[\raisebox{-5pt}{$\leftarrow$} ]
\end{ytableau}~.
\]
Evidently, when all parts of $\lambda'$ are even, we recover essentially the same upper hook product expressions as in the $\alpha = 1$ case. 
\begin{corollary}\label{cor:doublyeven}
For all $\lambda \vdash n$ such that each part of $\lambda'$ is even, let $\mu$ be the partition obtained from $\lambda$ by removing all rows of even index, and let $\mu = (a_1,\ldots,a_d~|~b_1,\ldots,b_d)$. Then 
	\begin{align*}
		\eta^\lambda_2 
		&= (-1)^{|\mu|} \sum_{i \leq \mu_i+1} (-1)^{\mu_i}  p_{\mu_1, a_1 - a_i}^{(2)} ~H_{i}^+(\mu),
	\end{align*}
	where $p_{m,k}^{(2)}$ is the probability that an element of $B_m$ has precisely $k$ fixed points.
\end{corollary}
\noindent As an aside, we note that similar results can be shown for all $\alpha \in \mathbb{N}_{+}$ by considering the rencontres numbers associated with the group $S_\alpha \wr S_n$, and in these cases one can derive determinantal formulas as we did in the previous section \emph{mutatis mutandis}.

We conclude this section with a somewhat more complicated formula for the eigenvalues of the perfect matching derangement graph in terms of extended lower hook products. To see why we should not immediately expect an expression as simple as the $\alpha = 1$ case for all $\lambda$, it is instructive to consider the one-row case:
\[
	\eta_2^{(n)} =  |\mathcal{D}_{2n}'| = \sum_{k=0}^n (-1)^k p_{n,k}^{(2)} (2k-1)!! (2(n-k)-1)!!. 
\]
Recall that for $\alpha = 1$, this summation is just a single term $\eta_1^{(n)} = d_{n,0}n!$. Indeed, one expects a more involved expression for $\alpha = 2$, due to the fact that even though we have $|\mathcal{D}_{2n}| = (2n-1)!!(1/\sqrt{e} + o(1))$, the $o(1)$ term does not converge to 0 nearly as fast as in the case of permutation derangements. In particular, the $n$th partial sum of the expansion of $e^{-1/2}$ is just an approximation of the probability of drawing a derangement from $\mathcal{M}_{2n}$.

In light of the lattice path interpretation given above, it will be useful to think of the parts of $\lambda$ as being grouped into consecutive pairs $\lambda_{2i-1},\lambda_{2i}$ where the difference $\lambda_{2i-1} - \lambda_{2i}$ between consecutive rows gives the order of the approximation, roughly speaking. By default, if $\lambda$ has less than $2i$ parts, then we set $\lambda_{2i} := 0$.
We define a shifted analogue of the extended lower hook products as follows
\[
	H^i_+(\lambda,j) := H^i_*(\lambda,j)H^i_*(\lambda^c,j),
\]
i.e., the product obtained by subtracting each factor of $H^i_+(\lambda)$ by $\alpha j$. 
\begin{theorem}[Eigenvalues of $\Gamma_{n,2}$]\label{thm:eigen2} For all $\lambda \vdash n$, we have
	\begin{align*}
		\eta^\lambda_2 
		&= (-1)^n \!\!\!\!\!\!\!\! \sum_{\substack{ i = 1 \\ 2i-1 \leq \lambda_{2i-1} + 1}} \!\!\!\!\!\!\!\! (-1)^{\lambda_{2i-1}} \!\!\!\! \sum_{\substack{j = 0 \\ 2i - 1 + j \leq \lambda_1}}^{\lambda_{2i-1} - \lambda_{2i}} \!\!\!\! (-1)^{j} ~p^{(2)}_{\lambda_1,a_1-a_i+j} ~ H^{2i-1}_+(\lambda,j).
	\end{align*}
\end{theorem}
\noindent One of the main obstacles towards getting an expression identical to the $\alpha = 1$ case is that the probability distribution $\{p_{n,i}^{(2)}\}$ is defined over the $(2n)!!$ elements of $B_n$, not the $(2n-1)!!$ elements of $S_{2n}/B_n \cong \mathcal{M}_{2n}$. Ideally, we seek a probability distribution $\{p_{n,i}'\}$ over perfect matchings such that $p_{n,i}'$ is the probability of drawing uniformly at random a perfect matching from $\mathcal{M}_{2n}$ that has $i$ edges in common with $m^*$. It seems that exact formulas for these probabilities cannot be expressed as succinctly as in the case of permutations (see the discussion above as well as the proof of Proposition~\ref{prop:lb}, for example). We leave it as an open question whether there is a more elegant formula for the $\alpha = 2$ case; nevertheless, we have given a closed form that is suitable for calculation and applications, as we have demonstrated in the previous sections.

\section{Future Work and Open Questions}

It may be worthwhile to study the Jack derangements, colored permutations, and colored perfect matchings from a purely combinatorial point of view. Indeed, one can verify that many of the well-known identities for derangements admit Jack analogues, and a closer study of their combinatorics may give more elegant formulas for the Jack derangements. 

In~\cite{Sniady19}, \'Sniady studies the Jack characters from the viewpoint of \emph{asymptotic representation theory}. Like the classical derangements, our expressions for the Jack derangements are quite amenable to asymptotic analysis, so it seems natural to consider the asymptotics of the Jack derangements and how they relate to \'Sniady's results (see also~\cite{CuencaDM23,DolegaS19}).

As discussed earlier, a byproduct of our main results at $\alpha = 1$ is a simple combinatorial form for $\text{Imm}_\lambda(J-I)$, which begs the question of whether other adjacency matrices have immanants with nice combinatorial properties. In particular, can one find nice combinatorial expressions for the immanants of adjacency matrices $A(G_n)$ of graph families $\{G_n\}$ besides $A(K_n) = J_n - I_n$? We refer the reader to~\cite{TessierMS} for more details on combinatorial interpretations of immanants. Along these lines, it would be quite interesting to find other unions of conjugacy classes $S \subseteq S_n$ such that the $\lambda$-eigenvalues of the normal Cayley graph $\text{Cay}(S_n,S)$ are counted by some ``$\lambda$-colored variant" of $S$.

Let $GL(n,q)$ be the group of $n \times n$ invertible matrices of $\mathbb{F}_q^{n \times n}$. We say that $g \in G$ is \emph{eigenvalue-free} if $\det(\lambda I - g) \neq 0$ for all $\lambda \in \mathbb{F}_q$~(see~\cite{NeumannP98} for a more detailed discussion). One can view such elements as a $q$-analogue of the derangements of $S_n$, and since the set of eigenvalue-free elements is a union of conjugacy classes of $GL(n,q)$, the eigenvalues of the normal Cayley graph generated by eigenvalue-free elements can be understood via the character theory of $GL(n,q)$. Like the symmetric group, there exists a characteristic map from the class algebra onto a Hopf algebra that allows one to get concrete (albeit extremely complicated) expressions for the irreducible characters of $GL(n,q)$ via symmetric function manipulations (see~\cite{GrinbergR14}, for example). In particular, a basis for this Hopf algebra can be defined in terms of \emph{Macdonald polynomials}, which can be seen as a $q$-analogue of the Jack polynomials~\cite[Ch.~VI]{Macdonald95}. A first step towards a full $q$-analogue of our main results would be to generalize what we have done here to Macdonald polynomials. Perhaps the combinatorics that arise in this work may give some insight as to what the right generalization should be.

\bibliographystyle{plain}
\bibliography{../research/master.bib}

\end{document}